\documentclass{mjm}

\IfFileExists{META}{\input{META}}{\newcommand{\mjmfirstpage}{99999}\newcommand{\mjmlastpage}{99999}}

\usepackage{amsmath}
\usepackage{amscd}
\usepackage{amssymb}
\usepackage{latexsym,nicefrac}
\usepackage{mjm}

\input xy
\xyoption{all}

\newcommand{\cO}{\mathcal O}

\renewcommand{\top}{\rm top}
\renewcommand{\dim}{{\rm dim}\,}
\renewcommand{\ker}{{\rm ker}\,}
\newcommand{\coker}{{\rm coker}\,}

\newcommand{\ev}{{\rm ev}}

\newcommand{\Si}{{\rm Sink}}

\newcommand{\colim}{\operatornamewithlimits{colim}}

\newcommand{\tor}{\mathrm{Tor}}
\newcommand{\tp}{t_+}
\newcommand{\tm}{t_{-}}
\DeclareMathOperator{\cofi}{hocofiber}

\newcommand{\sotimes}{\bar{\otimes}}
\newcommand{\comment}[1]{}  

\def\C{\mathbb{C}}

\def\Q{\mathbb{Q}}

\def\Z{\mathbb{Z}}

\newcommand\triqui{\vartriangleleft}

\def\fB{\mathfrak{B}}
\def\fA{\mathfrak{A}}
\numberwithin{equation}{section}
\theoremstyle{plain}
\newtheorem{thm}[equation]{Theorem}
\newtheorem{cor}[equation]{Corollary}
\newtheorem{lem}[equation]{Lemma}
\newtheorem{prop}[equation]{Proposition}

\theoremstyle{definition}

\theoremstyle{remark}
\newtheorem{rem}[equation]{Remark}
\newtheorem{exa}[equation]{Example}
\newtheorem{notation}{Notation} [equation]
\newtheorem*{ack}{Acknowledgement}
\begin{document}
\bibliographystyle{plain}

\author{ P. Ara}
\author{ M. Brustenga}
\author{G. Corti\~nas}
\headauthor{Ara, Brustenga, Corti\~nas}
\title{$K$-theory of Leavitt path algebras}
\headtitle{Leavitt path algebras}
\thanks{The first and second named authors were partially supported by
DGI MICIIN-FEDER MTM2008-06201-C02-01, and by the Comissionat per
Universitats i Recerca de la Generalitat de Catalunya. The third
named author was supported by CONICET and partially supported by grants PICT
2006-00836, UBACyT X051, and MTM2007-64074.}
\address{P. Ara, M. Brustenga\\
Departament de Matem\`atiques\\ Universitat Aut\`onoma de
Barcelona\\ 08193 Bellaterra (Barcelona), Spain}
\email{para@mat.uab.cat, mbrusten@mat.uab.cat}
\address{G. Corti\~nas\\ Dep. Matem\'atica\\ Ciudad Universitaria Pab 1\\
1428 Buenos Aires, Argentina}
\email{gcorti@dm.uba.ar}\urladdr{http://mate.dm.uba.ar/\~{}gcorti}

\date{\today}

\begin{abstract}
Let $E$ be a row-finite quiver and let $E_0$ be the set of vertices
of $E$; consider the adjacency matrix
$N'_E=(n_{ij})\in\Z^{(E_0\times E_0)}$, $n_{ij}=\#\{$ arrows from
$i$ to $j\}$. Write $N^t_E$ and $1$ for the matrices $\in
\Z^{(E_0\times E_0\setminus\Si(E))}$ which result from $N'^t_E$ and
from the identity matrix after removing the columns corresponding to
sinks. We consider the $K$-theory of the Leavitt algebra
$L_R(E)=L_\Z(E)\otimes R$. We show that if $R$ is either a
Noetherian regular ring or a stable $C^*$-algebra, then there is an
exact sequence ($n\in\Z$)
\[
\xymatrix{K_n(R)^{(E_0\setminus\Si(E))}\ar[r]^(.6){1-N_E^t}& K_n(R)^{(E_0)}\ar[r]& K_n(L_R(E))\ar[r]&
K_{n-1}(R)^{(E_0\setminus\Si(E))}}
\]
We also show that for general $R$, the obstruction for having a sequence as above is measured by twisted
nil-$K$-groups. If we replace $K$-theory by homotopy algebraic $K$-theory, the obstructions dissapear, and
we get, for every ring $R$, a long exact sequence
\begin{multline*}
KH_n(R)^{(E_0\setminus\Si(E))}\stackrel{1-N_E^t}{\longrightarrow}KH_n(R)^{(E_0)}\to
KH_n(L_R(E))\to KH_{n-1}(R)^{(E_0\setminus\Si(E))}
\end{multline*}
We also compare, for a $C^*$-algebra $\fA$, the algebraic $K$-theory of $L_\fA(E)$ with the topological
$K$-theory of the Cuntz-Krieger algebra $C^*_\fA(E)$. We show that the map
\[
K_n(L_\fA(E))\to K^{\top}_n(C^*_\fA(E))
\]
is an isomorphism if $\fA$ is stable and $n\in\Z$, and
also if $\fA=\C$, $n\ge 0$, $E$ is finite with no sinks, and $\det(1-N_E^t)\ne 0$.
\end{abstract}

\maketitle

\section{Introduction}
We consider the $K$-theory of the Leavitt algebra
$L_R(E)=L_\Z(E)\otimes R$ of a row-finite quiver $E$ with
coefficients in a ring $R$. To state our results, we need some
notation. Let $E_0$ be the set of vertices of $E$; consider the
adjacency matrix $N'_E=(n_{ij})\in\Z^{(E_0\times E_0)}$,
$n_{ij}=\#\{$ arrows from $i$ to $j\}$. Write $N^t_E$ and $1$ for
the matrices $\in \Z^{(E_0\times E_0\setminus\Si(E))}$ which result
from $N'^t_E$ and from the identity matrix after removing the
columns corresponding to sinks. Our results relate the $K$-theory of
$L_R(E)$ with the spectrum
\[
C=\cofi(K(R)^{(E_0\setminus \Si (E))}\overset{1-N_E^t}\longrightarrow
K(R)^{(E_0)})
\]
In terms of homotopy groups, the fundamental property of $C$ is that there is a long exact sequence $(n\in\Z)$
\begin{equation}\label{intro:cles}
\xymatrix{K_n(R)^{(E_0\setminus\Si(E))}\ar[r]^(.6){1-N_E^t}& K_n(R)^{(E_0)}\ar[r]& \pi_n(C)\ar[r]&
K_{n-1}(R)^{(E_0\setminus\Si(E))}}
\end{equation}
For a rather general class of rings (which includes all unital ones) and all row-finite quivers $E$, we show
(Theorem \ref{row-finitecase}) that there is a naturally split injective map
\begin{equation}\label{intro:map}
\pi_*(C)\to K_*(L_R(E))
\end{equation}
The cokernel of \eqref{intro:map}
can be described in terms of twisted nil-$K$-groups (see \ref{thm:skewle}, \ref{rem:coker}). We show that these
nil-$K$-groups vanish for some classes of rings $R$, including the following two cases:
\begin{itemize}
\item $R$ is a regular supercoherent ring (see \ref{rf-coh}). In
particular this covers the case where $R$ is a Noetherian regular
ring.
\item $R$ is a stable $C^*$-algebra (see \ref{cor:stablereg}).
\end{itemize}
In particular for such $R$ we get a long exact sequence
\begin{multline}\label{intro:les}
K_n(R)^{(E_0\setminus\Si(E))}\stackrel{1-N_E^t}{\longrightarrow} K_n(R)^{(E_0)}\to K_n(L_R(E))\to
K_{n-1}(R)^{(E_0\setminus\Si(E))}
\end{multline}
We also consider Weibel's homotopy algebraic $K$-theory
$KH_*(L_R(E))$. We show in \ref{thm:kh} that for any ring $R$ and
any row-finite quiver, there is a long exact sequence
\begin{multline}\label{intro:khles}
KH_n(R)^{(E_0\setminus\Si(E))}\stackrel{1-N_E^t}{\longrightarrow}KH_n(R)^{(E_0)}\to
KH_n(L_R(E))\\ \to KH_{n-1}(R)^{(E_0\setminus\Si(E))}\nonumber
\end{multline}
There is a natural comparison map $K_*\to KH_*$; if $R$ is a regular
supercoherent ring or a stable $C^*$-algebra, then $K_*(R)\to
KH_*(R)$ and $K_*(L_R(E))\to KH_*(L_R(E))$ are isomorphisms, so the
sequences agree in these cases. We further compare, for a
$C^*$-algebra $\fA$, the algebraic $K$-theory of $L_\fA(E)$ with the
topological $K$-theory of the Cuntz-Krieger algebra $C_\fA^*(E)$; we
show that the natural map
\[
\gamma^\fA_n(E):K_n(L_\fA(E))\to K_n(C_\fA^*(E))\to K^{\top}_n(C_\fA^*(E))
\]
is an isomorphism in some cases, including the following two:
\begin{itemize}
\item $\fA=\C$, $E$ is finite with no sinks, $\det(1-N^t_E)\ne 0$, and $n\ge 0$ (see \ref{thm:sus}).
\item $\fA$ is stable, $E$ is row-finite, and $n\in\Z$ (see \ref{thm:stable}).
\end{itemize}

The rest of this paper is organized as follows.
In Section \ref{sec:huni} we recall the results of Suslin and Wodzicki on excision in $K$-theory and draw
some consequences which are used further on in the article. The most general result on excision in $K$-theory, due to Suslin \cite{sus}, characterizes those rings $A$ on which $K$-theory satisfies excision in terms of the vanishing of $\tor$ groups over the unitalization $\tilde{A}=A\oplus \Z$. Namely $A$ satisfies excision if and only if
\begin{equation}\label{intro:h'uni}
\tor^{\tilde{A}}_*(\Z,A)=0 \qquad (*\ge 0).
\end{equation}
We call a ring $A$ $H'$-unital if it satisfies \eqref{intro:h'uni}; if $A$ is torsion-free as an abelian group, this is the same as saying that $R$ is $H$-unital in the sense of Wodzicki \cite{wod}. We show in
Proposition \ref{prop:huni} that if $A$ is $H'$-unital and $\phi:A\to A$ is an automorphism, then the same is
true of both the twisted polynomial ring $A[t,\phi]$ and the twisted Laurent polynomial ring $A[t,t^{-1},\phi]$.
We recall that, for unital $A$, the $K$-theory of the twisted Laurent polynomials was computed in \cite{gray} and \cite{yao}. If $R$ is a unital ring and $\phi:R\to pRp$ is a corner isomorphism, the twisted Laurent polynomial ring is not defined, but the corresponding object is the corner skew Laurent polynomial ring $R[t_+,t_-,\phi]$ of \cite{skew}. In Section \ref{sec:yao} we use the results of \cite{yao} and of Section \ref{sec:huni} to compute the $K$-theory of $R\otimes A[t_+,t_-,\phi\otimes 1]$ for $(R,\phi)$ as above, and $A$ any nonunital algebra such that $R\otimes A$ is $H'$-unital (Theorem \ref{thm:skewyao}). In the next section we consider
 the relation between two possible ways of defining the incidence matrix of a finite quiver,
 and show that the sequences of the form \eqref{intro:cles} obtained with either of them are essentially equivalent (Proposition \ref{prop:MyN}). In Section \ref{sect:skew} we use
 the results of the previous sections to compute the $K$-theory of the Leavitt algebra of a finite quiver
 with no sources with coefficients in an $H'$-unital ring (Theorem \ref{thm:skewle}). The general case
 of row-finite quivers is the subject of Section \ref{sect:row-finite}. Our most general result is Theorem
 \ref{row-finitecase}, where the existence of the split injective map \eqref{intro:map} is proved for
 the Leavitt algebra $L_A(E)$ of a row-finite quiver $E$. In the latter theorem, $A$ is required to be
 either a ring with local units, or a $\Z$-torsion free $H'$-unital ring. In Section \ref{sec:regcoh}
 we specialize to the case of Leavitt algebras with regular supercoherent coefficient rings. We show that
 the sequence \eqref{intro:les} holds whenever $R$ is regular supercoherent (Theorem \ref{rf-coh}). For example this holds if $R$ is a field, since fields are regular supercoherent; this particular case,
for finite $E$, is used in \cite{fpres} to compute the $K$-theory of the algebra $Q_R(E)$ obtained from $L_R(E)$ after inverting all square matrices with coefficients in the path algebra $P_R(E)$ which are sent to invertible matrices by the augmentation map $P_R(E)\to R^{E_0}$.
Section \ref{sec:kh} is devoted to homotopy algebraic $K$-theory, $KH$. For a unital ring
$R$, a corner isomorphism $\phi:R\to pRp$, and a ring $A$, we compute the $KH$-theory of
$R\otimes A[t_+,t_-,\phi\otimes 1]$ (Theorem \ref{thm:khskewyao}). Then we use this to establish
the sequence \eqref{intro:khles} for any row finite quiver $E$ and any coefficient ring $A$ (Theorem
\ref{thm:kh}).
In the last section we compare the $K$-theory of the Leavitt algebra $L_\fA(E)$ with coefficients in a $C^*$-algebra $\fA$ with the topological $K$-theory of the corresponding Cuntz-Krieger algebra $C^*_\fA(E)$. In Theorem \ref{thm:ktop} we establish the spectrum-level version of the well-known calculation of the topological
$K$-theory of the Cuntz-Krieger algebra $C^*_\fA(E)$ of a row-finite quiver $E$ with coefficients in a
$C^*$-algebra $\fA$. Theorem \ref{thm:sus} shows that if $E$ is a finite quiver without sinks and
such that $\det(1-N_E^t)\ne 0$, then the natural map $\gamma_n^{\C}:K_n(L_\C(E))\to K_n^{\top}(C^*_\C(E))$
is an isomorphism for $n\ge 0$ and the zero map for $n\le -1$. In Theorem \ref{thm:stable} we show
that if $\fB$ is a stable $C^*$-algebra, then $\gamma_n^\fB$ is an isomorphism for all $n\in\Z$.

\section{$H'$-unital rings and skew polynomial extensions}\label{sec:huni}

Let $R$ be a ring and $\tilde{R}=R\oplus\Z$ its unitization. We say that $R$ is {\it $H'$-unital} if
\[
\tor^{\tilde{R}}_*(R,\Z)=0\qquad (*\ge 0).
\]
Note that, for any, not necessarily $H'$-unital ring $R$,
\[
\tor^{\tilde{R}}_*(\Z,R)=\tor^{\tilde{R}}_{*+1}(\Z,\Z)=\tor^{\tilde{R}}_*(R,\Z)\qquad  (*\ge 0)
\]
Thus all these $\tor$ groups vanish when $R$ is $H'$-unital; moreover, in that case we also have
\[
\tor_*^{\tilde{R}}(R,R)=0 \qquad (*\ge 1), \qquad \tor_0^{\tilde{R}}(R,R)=R^2=R.
\]
A right module $M$ over a ring $R$ is called {\it $H'$-unitary} if
$\tor^{\tilde{R}}_*(M,\Z )=0$. The definition of $H'$-unitary for
left modules is the obvious one.

\begin{exa}
If $R$ is $H'$-unital then it is both right and left $H'$-unitary as
a module over itself. Let $\phi:R\to R$ be an endomorphism. Consider
the  bimodule ${}_{\phi}R$ with left multiplication given by $a\cdot
x=\phi(a)x$ and the usual right multiplication. As a right module,
$_{\phi}R\cong R$, whence it is right $H'$-unitary. If moreover
$\phi$ is an isomorphism, then it is also isomorphic to $R$ as a
left module, via $\phi$, and is thus left $H'$-unitary too.
\end{exa}

\begin{rem}\label{rem:hh'}
The notion of $H'$-unitality is a close relative of the notion of $H$-unitality introduced by Wodzicki in
\cite{wod}. The latter notion depends on a functorial complex $C^\mathrm{bar}(A)$, the {\it bar complex} of $A$; we have $C_n^\mathrm{bar}(A)=A^{\otimes n+1}$. The ring $A$ is called $H$-unital if for all abelian groups $V$, the complex $C^\mathrm{bar}(A)\otimes V$ is acyclic. If $A$ is flat as a $\Z$-module, then $C^\mathrm{bar}(A)$ is a complex of flat $\Z$-modules and $H_*(C^\mathrm{bar}(A))=\tor_*^{\tilde{A}}(\Z,A)$. Hence $H'$-unitality is the same as $H$-unitality for rings which are flat as $\Z$-modules. Unital rings
 are both $H$ and $H'$-unital. Because $C^{\mathrm{bar}}$ commutes with filtering colimits, the class of $H$-unital rings is closed under such colimits. Similarly, there is also a functorial complex which computes $\tor^{\tilde{A}}(\Z,A)$ and which commutes with filtering colimits (\cite[6.4.3]{friendly}); hence also the class of $H'$-unital rings is closed under filtering colimits. If $A$ is $H$ or $H'$-unital then the same
 is true of the matrix ring $M_nA$. In the $H$-unital case, this is proved in \cite[9.8]{wod}; the $H'$-unital
 case follows from a theorem of Suslin cited below (Theorem \ref{thm:excisus}). The class of $H$-unital rings is furthermore closed under tensor products (\cite[7.10]{suswod}).
\end{rem}

\begin{lem}\label{lem:Qhuni}
Let $A$ be a ring. If $A$ is $H'$-unital, then $A\otimes\Q$ is $H'$-unital.
\end{lem}
\begin{proof}
Tensoring with $\Q$ over $\Z$ is an exact functor from
$\tilde{A}$-modules to $ \tilde{A}\otimes\Q$-modules which preserves
free modules. Hence if $L\to A$  is a free $\tilde{A}$-resolution,
then $L\otimes\Q\to A\otimes\Q$ is a free
$\tilde{A}\otimes\Q$-resolution. Moreover,

\[
\Q\otimes_{\tilde{A}\otimes\Q}L\otimes\Q=L\otimes\Q/A\cdot
L\otimes\Q=(L/A\cdot L)\otimes\Q
\]

Hence
\[
\tor_*^{\tilde{A}\otimes\Q}(A\otimes\Q,\Q)=\tor_*^{\tilde{A}}(\Z,A)\otimes\Q
\]
Thus $A$ $H'$-unital implies that
$0=\tor_*^{\tilde{A}\otimes\Q}(A\otimes\Q,\Q)$. But by \cite[\S
2]{wod},
$\tor_*^{\tilde{A}\otimes\Q}(A\otimes\Q,\Q)=H_*(C^{\mathrm{bar}}(A\otimes\Q))$.
Thus $A\otimes\Q$ is $H$-unital, and therefore $H'$-unital.
\end{proof}

\begin{cor}\label{cor:otimes} If $A$ and $B$ are $H'$-unital, and $B$ is a $\Q$-algebra, then $A\otimes B$ is $H'$-unital.
\end{cor}
\begin{proof}
It follows from the previous lemma and from the fact (proved in \cite[7.10]{suswod}) that the tensor product of $H$-unital $\Q$-algebras is $H$-unital.
\end{proof}
\begin{exa} The basic examples of $H'$-unital rings we shall be concerned with are unital rings and
$C^*$-algebras. The fact that the latter are $H'$-unital follows from the results of \cite{sus} and \cite{suswod} (see \cite[6.5.2]{friendly} and Theorem \ref{thm:excisus} below). If $A$ is an $H'$-unital ring and $\fB$ a $C^*$-algebra, then $A\otimes\fB$ is $H'$-unital, by Corollary \ref{cor:otimes}.
\end{exa}

A ring $R$ is said to {\it satisfy excision} in $K$-theory if for every
embedding $R\triqui S$ of $R$ as a two-sided ideal of a unital ring
$S$, the map $K(R)=K(\tilde{R}:R)\to K(S:R)$ is an equivalence. One
can show (see e.g. \cite[1.3]{kabi}) that if $R$ satisfies excision in
$K$-theory and $R$ is an ideal in a nonunital ring $T$, then the map
$K(R)\to K(T:R)$ is an equivalence too.
\goodbreak
The main result about $H'$-unital rings which we shall need is the following.

\begin{thm}\label{thm:excisus}
{\rm (\cite{sus})} A ring $R$ is $H'$-unital if and only if it
satisfies excision in $K$-theory.
\end{thm}

Using the theorem above we get the following Morita invariance result for $H'$-unital rings.

\begin{lem}\label{lem:morita} Let $R$ be a unital ring, $e\in R$ an idempotent. Assume $e$ is {\rm full}, that is, assume $ReR=R$.
Further let $A$ be a ring such that both $R\otimes A$  and $eRe\otimes A$ are $H'$-unital. Then the inclusion map $eRe\otimes A\to R\otimes A$ induces an equivalence
$K(eRe\otimes A)\to K(R\otimes A)$.
\end{lem}
\begin{proof}
Put $S=R\otimes \tilde{A}$, and consider the idempotent $f=e\otimes
1\in S$. One checks that $f$ is a full idempotent, so that
$K(fSf)\to K(S)$ is an equivalence. Now apply excision.
\end{proof}

\begin{prop}\label{prop:huni} Let $R$ be a ring and $\phi:R\to R$ an automorphism. Assume $R$ is $H'$-unital. Then
$R[t,\phi]$ and $R[t,t^{-1},\phi]$ are $H'$-unital rings.
\end{prop}
\begin{proof}
If $P$ is a projective resolution of $R$ as a right $\tilde{R}$-module,
then $P\otimes_{\tilde{R}} \widetilde{R[t,\phi]}$ is a complex of right $\widetilde{R[t,\phi]}$-
projective modules. Moreover, we have an isomorphism of $R$-bimodules
\[
\widetilde{R[t,\phi]}=\tilde{R}\oplus\bigoplus_{n=1}^\infty
Rt^n\cong \tilde{R}\oplus\bigoplus_{n=1}^\infty R_{\phi^n}
\]
Thus, because $R$ is assumed $H'$-unital,
\[
H_*(P\otimes_{\tilde{R}}R)=\tor_*^{\tilde{R}}(R,R)=\left\{\begin{matrix} 0 & *\ge 1\\ R & *=0\end{matrix}\right.
\]
Here we have  used only the left module structure of $R$; the
identities above are compatible with any right module structure, and
in particular with both the usual one and that induced by $\phi^n$.
It follows that
\[
Q=P\otimes_{\tilde{R}}\widetilde{R[t,\phi]}
\]
is a projective resolution of $R[t,\phi]$ as a right
$\widetilde{R[t,\phi]}$-module. Since $\tilde{R}\to
\widetilde{R[t,\phi]}$ is compatible with augmentations, we have
\[
Q\otimes_{\widetilde{R[t,\phi]}} \Z=P\otimes_{\tilde{R}}\Z
\]
Hence $R[t,\phi]$ is $H'$-unital. Next we consider the case of the skew Laurent polynomials. We have
a bimodule isomorphism
\[
\widetilde{R[t,t^{-1},\phi]}=\tilde{R}\oplus\bigoplus_{n=1}^\infty
(Rt^n\oplus t^{-n}R) \cong \tilde{R}\oplus\bigoplus_{n=1}^\infty
(R_{\phi^n}\oplus {}_{\phi^n} R)
\]
Thus since ${}_{\phi^n}R$ is left $H'$-unitary, the same argument as
above shows that $R[t,t^{-1},\phi]$ is $H'$-unital.
\end{proof}

\section{$K$-theory of twisted Laurent polynomials}\label{sec:yao}

Let $X$, $N_+$, $N_-$ and $Z$ be objects in a triangulated category
$\mathcal T$. Let $\phi:X\to X$ and $j^{\pm}\colon X\oplus
N_{\pm}\to Z$ be maps in $\mathcal T$. Let $i^{\pm}:X\to X\oplus
N_{\pm}$ be the inclusion maps. Define a map
\[
\psi=\left[\begin{matrix}i^+&i^+\\ i^{-}&
i^{-}\circ\phi\end{matrix}\right]: X\oplus X\to (X\oplus N_+)\oplus
(X\oplus N_-),
\]
\begin{lem}\label{lem:fibs}
Assume
\[
\xymatrix{X\oplus X\ar[r]^(.3){\psi} & (X\oplus N_+)\oplus (X\oplus
N_-)\ar[r]^(.8){[j^+,j^-]}& Z \ar[r]^(.3){\partial}  & \Sigma
X\oplus \Sigma X }
\]
is an exact triangle in $\mathcal T$. Then
\begin{equation}
\xymatrix{ X \ar[rr]^(.35){[0,1-\phi,0]}&& N_+\oplus X\oplus N_-\ar[rrr]^(.6){[j^+|_{N_+},-j^-|_{X},j^-|_{N_-}]} &&&Z \ar[r]^{\partial '}& \Sigma X}
\end{equation}
is an exact triangle in $\mathcal T$, for
suitable $\partial '$.
In particular,
\[ Z\cong
N_+\oplus N_-\oplus \text{\rm cone}(1-\phi:X\to X).
\]
\end{lem}
\begin{proof}
Note that
\[\psi=\left[\begin{matrix}1 &1\\ 0 & 0 \\ 1 & \phi \\
0 & 0\end{matrix}\right].\]
Consider the maps
\[
\psi _1=\left[\begin{matrix}1 &0 & 0 & 0 \\ 0 & 1 & 0 & 0 \\ 1 & 0 & -1 & 0 \\
0 & 0 & 0 & 1\end{matrix}\right], \psi _2= \left[\begin{matrix}1 & 0\\ 0 & 0 \\ 0 & 1-\phi \\
0 & 0\end{matrix}\right], \text{ and }\psi_3= \left[\begin{matrix}1 &1\\ 0
& 1\end{matrix}\right].
\]
Write $\psi =\psi _1\psi _2\psi _3$. Then we have an exact triangle
\[
\xymatrix{X\oplus X\ar[r]^(.3){\psi_2} & (X\oplus N_+)\oplus
(X\oplus N_-)\ar[r]^(.8){\tilde{j}} & Z\ar[r]^(.3){\partial ''} &
\Sigma X\oplus \Sigma X }
\]
Here $\tilde{j}=[j^+,j^-]\psi _1$ and $\partial ''=\psi _3\partial
$. The result follows. \end{proof}

Let $\phi:X\to X$ be a map of spectra. We write $\phi^{-1}X$ for the
colimit of the following direct system
\[
\xymatrix{X\ar[r]^\phi& X\ar[r]^\phi& X\ar[r]^\phi& X\ar[r]^\phi&\dots}
\]

\begin{lem}\label{lem:cofi} Let $X$ and $\phi$ be as above, and consider the map $\hat{\phi}:\phi^{-1}X\to \phi^{-1}X$
induced by $\phi$. Then
\[
\cofi(1-\phi:X\to X)\cong \cofi(1-\hat{\phi}:\phi^{-1}X\to \phi^{-1}X)
\]
\end{lem}
\begin{proof}

Write $\hat{\phi}:\phi^{-1}X\to \phi^{-1}X$ for the induced map; we
have a homotopy commutative diagram
\begin{equation}\label{triang}
\xymatrix{\phi^{-1}X\ar[r]^{1-\hat{\phi}}&\phi^{-1}X\ar[r]&\cofi(1-\hat{\phi})\\
          X\ar[r]^{1-\phi}\ar[u]&X\ar[u]\ar[r]&\cofi(1-\phi)\ar[u]^f }
\end{equation}
Both the top and bottom rows are fibration sequences. We have to
show that the map of stable homotopy groups
$f_n:\pi_n\cofi(1-\phi)\to \pi_n\cofi(1-\hat{\phi})$ induced by $f$
is an isomorphism. Denote by $\phi _n$ the endomorphism of $\pi
_n(X)$ induced by $\phi$.
 Note that $\phi_n$ induces a $\Z[t]$-action
on $\pi_nX$, and that
\[
\pi_n(\phi^{-1}X)=\Z[t,t^{-1}]\otimes_{\Z[t]}\pi_nX=:\phi_n^{-1}\pi_nX
\]
It follows that the long exact sequence of homotopy groups associated to the top fibration of \eqref{triang}
is the result of applying the functor $\Z[t,t^{-1}]\otimes_{\Z[t]}$ to that of the
bottom. In particular the left and right vertical maps in the diagram below are isomorphisms
\[
\xymatrix{0\to \coker(1-\hat{\phi}_n)\ar[r]& \pi_n(\cofi (1-\hat{\phi}))\ar[r]&\ker (1-\hat{\phi}_{n+1})\to 0\\
0\to \coker(1-\phi_n)\ar[u]\ar[r]& \pi_n(\cofi (1-\hat\phi))\ar[u]^{f_n}\ar[r]&\ker (1-\phi_{n+1})\to 0\ar[u]}
\]
It follows that $f$ is an equivalence, as wanted.
\end{proof}

It will be useful to introduce the following notation.

\begin{notation}
Let $A$ be a unital ring and let $\phi\colon A\to A$ be an
automorphism. Define $NK(A,\phi)_+=\cofi(K(A)\to K(A[t,\phi])$ and
$NK(A,\phi)_-=\cofi(K(A)\to K(A[t,\phi^{-1}])$. We have
$$K(A[t,\phi])=K(A)\oplus NK(A, \phi)_+\, , \qquad K(A[t,\phi ^{-1}])=K(A)\oplus NK(A,\phi)_-.$$
Now let $A$ be an arbitrary ring and let $\phi \colon A\to A$ be an
endomorphism. Write $B=\phi^{-1}A$ for the colimit of the inductive
system
\[
\xymatrix{A\ar[r]^\phi& A\ar[r]^\phi & A\ar[r]^\phi&\dots}
\]
Then $\phi$ induces an automorphism $\hat{\phi}:B\to B$ and we can
extend it to the unitization $\tilde{B}$. Put
$$NK(A,\phi)_+: =NK(\tilde{B},\hat{\phi})_+\, , \qquad NK(A,\phi)_-:=NK(\tilde{B},\hat{\phi})_-\, ,$$
so that $K(\tilde{B}[t,\hat{\phi}])=K(\tilde{B})\oplus NK(A,
\phi)_+$ and similarly for $K(\tilde{B}[t,\hat{\phi}^{-1}])$.
Observe that this definition of $NK(A,\phi)_{\pm}$ agrees with the
above when $A$ is unital and $\phi$ is an automorphism. Moreover we
have $NK(A,\phi)_{\pm}=NK(B,\hat{\phi})_{\pm}$.
\end{notation}
\begin{lem}\label{lem:nkhuni} Let $A$ be $H'$-unital, $\phi:A\to A$ an endomorphism, and $B=\phi^{-1}A$. Then
$K(B[t,\phi^{\pm 1}])\cong\phi^{-1}K(A)\oplus NK(A,\phi)_{\pm}$.
\end{lem}
\begin{proof} We have
an exact sequence
\[
0\to B[t,\hat{\phi}^{\pm 1}]\to \widetilde{B[t,\hat{\phi}^{\pm 1}]}\to \Z[t]\to 0
\]
By Proposition \ref{prop:huni}, the ring $B[t,\hat{\phi}^{\pm 1}]$
is $H'$-unital. Hence $K(B[t,\hat{\phi}^{\pm 1}])=K(B)\oplus
NK(B,\phi)_\pm$, by excision. Next, the fact that $K$-theory
preserves filtering colimits (see \cite[IV.6]{Wei} for the unital
case; the non-unital case follows from the unital case by using that
unitization preserves colimits --because it has a right adjoint--
and that $ K(A)=K(\tilde{A}:A)$) implies that
$K(B)\cong\phi^{-1}K(A)$.
\end{proof}

We shall make use of the construction of the corner skew Laurent
polynomial ring $S[t_+,t_-,\phi]$, for a corner-isomorphism $\phi
\colon S\to pSp$; see \cite{skew}.

\begin{thm}\label{thm:skewyao}
Let $R$ be a unital ring and let $A$ be a ring. Let
$\phi:R\to pRp$ be a corner-isomorphism. Assume that $R\otimes A$ is $H'$-unital. Then there is a homotopy
fibration of nonconnective spectra
\[
\xymatrix{K(R\otimes A ) \ar[r]^(.2){1-\phi\otimes 1}&K(R\otimes
A)\oplus NK(R\otimes A,\phi\otimes 1)_+\oplus NK(R\otimes
A,\phi\otimes 1)_-} \]
\[ \xymatrix{ {}  \ar[r] & K((R\otimes
A)[\tp,\tm,\phi\otimes 1])}
\]
In other words, \begin{align*}  K((R\otimes A)[\tp,\tm,\phi\otimes
1]) & = NK(R\otimes A,\phi\otimes 1)_+\oplus NK(R\otimes
A,\phi\otimes 1)_- \\ & \oplus \cofi(K(R\otimes
A)\overset{1-\phi\otimes 1}\longrightarrow K(R\otimes A)).
\end{align*}
\end{thm}
\begin{proof}
\noindent{Step $1$:} Assume that $\phi$ is a unital isomorphism and
$A=\Z$. In this case the skew Laurent polynomial ring is the crossed
product by $\Z$; $R[\tp,\tm,\phi]=R[t,t^{-1},\phi]$. Let
$i^{\pm}:R\to R[t_{\pm},\phi]$ and $j^\pm:R[t_\pm,\phi]\to
R[\tp,\tm,\phi]$ be the inclusion maps. By the proof of
\cite[Theorem 2.1]{yao}, there is a homotopy fibration
\[
\xymatrix{K(R)\oplus K(R)\ar[r]^(.4){\psi}&K(R[\tp,\phi]) \oplus
K(R[\tm,\phi])\ar[r]^(.6){[j^+,j^-]}& K(R[\tp ,\tm ,\phi])}
\]
and $K(R[t_{\pm},\phi])=K(R)\oplus NK(R,\phi)_{\pm}$. Here
\[
\psi=\left[\begin{matrix}i^+&i^+\\ i^-& i^-\circ\phi\end{matrix}\right]
\]
Application of Lemma \ref{lem:fibs} yields the fibration of the theorem; this finishes the case when $\phi$
is a unital isomorphism. \\

\noindent{Step $2$:} Assume that $B$ is an $H'$-unital ring and that
$\phi\colon B\to B$ is an isomorphism. Then by the previous step,
the augmentation $\tilde{B}\to \Z$ induces a map of fibration
sequences
\[
\xymatrix{K(\tilde{B})\ar[r]^(.25){1-\tilde{\phi}}\ar[d]&K(\tilde{B})
\oplus NK(\tilde{B},\tilde{\phi})_+\oplus
NK(\tilde{B},\tilde{\phi})_-\ar[r]\ar[d]&
K(\tilde{B}[\tp,\tm,\tilde{\phi}])\ar[d]\\
K(\Z)\ar[r]_0 &K(\Z) \ar[r] & K(\Z[t,t^{-1}])}
\]
Since $B[t_\pm,\phi]$ and $B[\tp,\tm,\phi]$ are $H'$-unital by
Proposition \ref{prop:huni},
the fibers of the vertical maps give the fibration of the theorem.\\

\noindent{Step $3$:} Assume that $R$ is unital and let $\phi$ be a
corner isomorphism. Let $A$ be an $H'$-unital ring. Write
$S=\phi^{-1}R$ for the colimit of the inductive system
\[
\xymatrix{R\ar[r]^\phi& R\ar[r]^\phi & R\ar[r]^\phi&\dots}
\]
Then $\phi$ induces an automorphism $\hat{\phi}:S\to S$.
Set $R_n=R$; then $B=S\otimes A=\colim_nR_n\otimes A$ is $H'$-unital, since $R_n\otimes A$ is $H'$-unital
by hypothesis, and $H'$-unitality is preserved under filtering colimits (see Remark \ref{rem:hh'}).
Since $\hat{\phi}\otimes 1$ is an
automorphism of $B$, Step $2$ gives
\begin{align}\label{yaob}
K(B[\tp,\tm,\hat{\phi}\otimes 1])=&\cofi (1-\hat{\phi}\otimes
1:K(B)\to K(B))\\ &\oplus NK(B,\hat{\phi}\otimes 1)_+\oplus
NK(B,\hat{\phi}\otimes 1)_- \, .\nonumber
\end{align}
Because $K$-theory commutes with filtering colimits,
 we have $K(B)=(\phi\otimes 1)^{-1}K(R\otimes A)$.
 Thus by Lemma
\ref{lem:cofi},
\begin{multline}\label{1-fib}
\cofi (1-\hat{\phi}\otimes 1:K(B)\to K(B))\cong\\
\cofi (1-\phi\otimes
1 :K(R\otimes A)\to K(R\otimes A))
\end{multline}
Write $\varphi _n\colon R_n\to S$ for the canonical map of the
colimit, and put $e_n=\varphi _n(1)$.
For $n\ge 0$, there is a ring isomorphism $\psi_n\colon (R\otimes
A)[\tp,\tm,\phi\otimes 1]\to (e_n\otimes
1)B[t,t^{-1},\hat{\phi}\otimes 1](e_n\otimes 1)$, where $e_n\otimes
1\in S\otimes \tilde{A}$, such that $\psi_n (r\otimes a)=\varphi
_n(r)\otimes a$, and $\psi_n(\tp)=(e_n\otimes 1)t(e_n\otimes 1)$,
and $\psi_n(\tm)=(e_n\otimes 1)t^{-1}(e_n\otimes 1)$.

Consider the map $\eta:(R\otimes A)[\tp,\tm,\phi\otimes
1]\longrightarrow (R\otimes A)[\tp,\tm,\phi\otimes 1]$, $\eta(x)=\tp
x\tm$. There is a commutative diagram
\begin{equation}
\begin{CD}
(R\otimes A)[\tp,\tm,\phi\otimes 1] @>{\psi_n}>> (e_n\otimes
1)B[t,t^{-1},
\hat{\phi}\otimes 1](e_n\otimes 1) \\
@V{\eta}VV  @V{i}VV \\
(R\otimes A)[\tp,\tm,\phi\otimes 1] @>{\psi_{n+1}}>> (e_{n+1}\otimes
1)B[t,t^{-1},\hat{\phi}\otimes 1](e_{n+1}\otimes 1)
\end{CD}
\end{equation}
It follows that $B[\tp,\tm,\hat{\phi}\otimes 1]=\eta^{-1}(R\otimes
A)[\tp,\tm,\phi\otimes 1]$. Hence we have
$K(B[\tp,\tm,\hat{\phi}\otimes])\cong \eta^{-1}K((R\otimes
A)[\tp,\tm,\phi\otimes 1])$. But since $\tm\tp=1$, the map $\eta$
induces the identity on $K((R\otimes A)[\tp,\tm,\phi\otimes 1])$
(e.g. by \cite[2.2.6]{friendly}). Thus
\begin{equation}\label{kakb}
K((R\otimes A)[\tp,\tm,\phi\otimes 1])\cong
K(B[\tp,\tm,\hat{\phi}\otimes 1])
\end{equation}
In addition we have
\begin{equation}\label{nkcoli}
NK(R\otimes A,\phi\otimes 1)_{\pm}\cong NK(B,\hat{\phi}\otimes
1)_{\pm}.
\end{equation}
Rewrite \eqref{yaob} using \eqref{1-fib}, \eqref{nkcoli} and \eqref{kakb} to finish the third (and final) step.\\
\end{proof}

\section{Matrices associated to finite quivers}\label{sect:matrix}

Let $E$ be a finite quiver. Write $E_0$ for the set of vertices and $E_1$
for the set of arrows. In this section we assume both $E_0$ and $E_1$ are finite, of
cardinalities $e_0$ and $e_1$. If $\alpha\in E_1$, we write
$s(\alpha)$ for its source vertex and $r(\alpha)$ for its range.
There are two matrices with nonnegative integer coefficients
associated with $E$; these are best expressed in terms of the {\it range} and {\it source} maps
$r,s:E_1\to E_0$. If $f\colon E_1\to E_0$ is a map of finite sets,
and $\chi_x$, $\chi_y$ are the characteristic functions of $\{x\}$
and $\{y\}$, we write
\begin{gather*}
f^*\colon \Z^{E_0}\to \Z^{E_1},\qquad f^*(\chi_y)=\sum_{f(x)=y}\chi_x\\
f_*\colon \Z^{E_1}\to \Z^{E_0},\qquad f_*(\chi_x)=\chi_{f(x)}.
\end{gather*}
Put
\begin{equation}\label{MyN}
M_E=r^*s_*\qquad N_E'=s_*r^*
\end{equation}
We identify these homomorphisms with their matrices with respect to
the canonical basis. The matrices $M_E=[m_{\alpha,\beta}]\in
M_{e_1}\Z$ and $N_E'=[n_{i,j}]\in M_{e_0}\Z$ are given by
\begin{gather*}
m_{\alpha,\beta}=\delta_{r(\alpha),s(\beta)} \\
n_{i,j}=\#\{\alpha\in E_1:s(\alpha)=i,\quad r(\alpha)=j\}
\end{gather*}
For $i=0,1$, we consider the chain complex $C^i$ concentrated in
degrees $0$ and $1$, with $C^i_j=\Z^{e_i}$ if $j=0,1$, and with
boundary map $1-N_E'$ if $i=0$ and $1-M_E$ if $i=1$. Pictorially
\begin{gather}
C^0: \xymatrix{\Z^{E_0}\ar[r]^{1-N_E'} &\Z^{E_0}}\nonumber\\
C^1: \xymatrix{\Z^{E_1}\ar[r]_{1-M_E} &\Z^{E_1}}\label{cplx}
\end{gather}
\begin{lem}\label{lem:htpy}
The maps $r^*$ and $s_*$ induce inverse homotopy equivalences  $C^0\leftrightarrows C^1$.
\end{lem}
\begin{proof}
Straightforward.
\end{proof}

\begin{prop}\label{prop:MyN}
If $X$ is a spectrum, then $\cofi (1-M_E:X^{e_1}\to X^{e_1})\cong
\cofi (1-N_E':X^{e_0}\to X^{e_0})$.
\end{prop}
\begin{proof}
Note  $r^*$ induces a map
\[
\xymatrix{X^{e_0}\ar[r]^{1-N_E'}\ar[d]_{r^*}& X^{e_0}\ar[d]_{r^*}\ar[r] &\cofi(1-N_E')\ar@{.>}[d]^f\\
X^{e_1}\ar[r]^{1-M_E} &X^{e_1}\ar[r] &\cofi(1-M_E)}
\]
From the long exact sequences of homotopy groups of the fibrations above, we obtain

\[
\xymatrix{0\ar[d] & 0\ar[d]\\
H_0(C^0\otimes \pi_n(X))\ar[r]^{r^*}\ar[d]&H_0(C^1\otimes \pi_n(X))\ar[d]\\
\pi_n\cofi(1-N_E')\ar[r]^f\ar[d]& \pi_n\cofi(1-M_E)\ar[d]\\
H_1(C^0\otimes\pi_{n-1}X)\ar[r]^{r^*}\ar[d]& H_1(C^1\otimes\pi_{n-1}X)\ar[d]\\
0 & 0}
\]

By Lemma \ref{lem:htpy}, the horizontal maps at the two extremes are isomorphisms; it follows that the map
in the middle is an isomorphism too.
\end{proof}

Recall that a vertex $i\in E_0$ is called a {\it source}
(respectively, a {\it sink}) in case $r^{-1}(i)=\emptyset$
(respectively, $s^{-1}(i)=\emptyset$). We will denote by $\Si (E)$
the sets of sinks of $E$.

\section{$K$-theory of the Leavitt algebra I: finite quivers without sinks}\label{sect:skew}

Let $E$ be a finite quiver and $M=M_E$.
The {\it path ring of $E$} is the ring $P=P_\Z(E)$ with one
generator for each arrow $\alpha\in E_1$ and one generator $p_i$ for
each vertex $i\in E_0$, subject to the following relations
\begin{gather}
                      p_ip_j=\delta_{i,j}p_i\, , \qquad (i,j\in E_0)\label{piortog}\\
                      p_{s(\alpha)}\alpha =\alpha =\alpha p_{r(\alpha )}\, ,\qquad (\alpha\in E_1)\label{alpharel}
\end{gather}

The ring $P$ has a basis formed by the $p_i$, the $\alpha$, and
the products $\alpha_1\cdots\alpha_n$ with
$r(\alpha_i)=s(\alpha_{i+1})$. We think of these as paths in the
quiver, of lengths, $0$, $1$ and $n$, respectively. Observe that $P$
is unital, with $1=\sum _{i\in E_0}p_i$.

Consider the {\it opposite quiver} $E^*$; this is the quiver with the same sets
of vertices and arrows, but with the range and source functions switched.
Thus $E_i^*=E_i$ $(i=0,1)$ and if we write $\alpha^*$ for
the arrow $\alpha\in E_1$ considered as an arrow
 of $E^*$, we have $r(\alpha^*)=s(\alpha)$ and $s(\alpha^*)=r(\alpha)$. The path
ring $P^*=P(E^*)$ is generated by the $p_i$ $(i\in E_0)$
 and the $\alpha^*\in E^*_1$; the relation \eqref{piortog} is satisfied, and we also have
 \begin{equation}\label{alpha^*rel}
p_{r(\alpha)}\alpha ^*=\alpha^* =\alpha^*p_{s(\alpha)}\, , \qquad
(\alpha\in E_1).
 \end{equation}

The {\it Leavitt path ring of $E$} is the ring $L=L_\Z(E)$
on generators $p_i$ $(i\in E_0)$, $\alpha\in E_1$, and $\alpha^*\in
E^*_1$, subject to relations  \eqref{piortog}, \eqref{alpharel}, and
\eqref{alpha^*rel}, and to the following two additional relations
\begin{gather}
\alpha^*\beta=\delta_{\alpha,\beta}p_{r(\alpha)}\label{alpha^*alpha}\\
p_i=\sum_{s(\alpha)=i}\alpha\alpha^* \qquad (i\in E_0\setminus \Si
(E)) \label{alphaalpha^*}
\end{gather}
From these last two relations we obtain
\begin{align}
\alpha^*\alpha&=\sum_{s(\beta)=r(\alpha)}\beta\beta^*\nonumber\\
&=\sum_{\beta\in E_1}m_{\beta,\alpha}\beta\beta^*\label{ck}
\end{align}
It also follows, in case $E$ has no sinks, that the
$q_\beta=\beta\beta^*$ are a complete system of orthogonal
idempotents; we have
\begin{equation}\label{qbeta}
\sum_{\beta\in E_1}q_\beta=1,\qquad q_\alpha q_\beta=\delta_{\alpha,\beta}q_{\beta}
\end{equation}

The ring $L$ is equipped with an involution and a $\Z$-grading.
The involution $x\mapsto x^*$ sends $\alpha\mapsto \alpha^*$ and
$\alpha^*\mapsto \alpha$. The grading is determined by $|\alpha|=1$,
$|\alpha^*|=-1$. By \cite[proof of Theorem 5.3]{AMP}, we have
$L_0=\bigcup _{n=0}^{\infty} L_{0,n}$, where $L_{0,n}$ is the linear
span of all the elements of the form $\gamma \nu^*$, where $\gamma$
and $\nu$ are paths with $r(\gamma)=r(\nu)$ and $|\gamma |=|\nu
|=n$. For each $i$ in $E^0$, and each $n\in \Z^+$, let us denote by
$P(n,i)$ the set of paths $\gamma$ in $E$ such that $|\gamma |=n$
and $r(\gamma)=i$.  The ring $L_{0,0}$ is isomorphic to $\prod
_{i\in E^0}k$. In general the ring $L_{0,n}$ is isomorphic to
$$\big[ \prod _{m=0}^{n-1}\big( \prod _{i\in \Si (E)}M_{|P(m,i)|}(\Z)\big)\big]
\times \big[ \prod _{i\in E_0}M_{|P(n,i)|}(\Z) \big] . $$ The
transition homomorphism $L_{0,n}\to L_{0,n+1}$ is the identity on
the factors $\prod _{i\in \Si (E)}M_{|P(m,i)|}(\Z)$, for $0\le m\le
n-1$, and also on the factor $\prod _{i\in \Si (E)}M_{|P(n,i)|}(\Z)$
of the last term of the displayed formula. The transition
homomorphism
$$\prod_{i\in E_0\setminus \Si (E)}M_{|P(n,i)|}(\Z)\to \prod_ {i\in
E_0}M_{|P(n+1,i)|}(\Z)$$ is a block diagonal map induced by the
following identification in $L(E)_0$: A matrix unit in a factor
$M_{|P(n,i)|}(\Z)$, where $i\in E_0\setminus \Si (E)$, is a monomial
of the form $\gamma\nu ^*$, where $\gamma$ and $\nu$ are paths of
length $n$ with $r(\gamma)=r(\nu)=i$. Since $i$ is not a sink, we
can enlarge the paths $\gamma$ and $\nu$ using the edges that $i$
emits, obtaining paths of length $n+1$, and relation
(\ref{alphaalpha^*}) in the definition of $L(E)$ gives
\[
\gamma \nu^*=\sum _{\{\alpha\in E_1\mid s(\alpha )=i\}}(\gamma
\alpha)(\nu\alpha)^*.
\]
Assume $E$ has no sources. For each $i\in E_0$, choose an arrow
$\alpha_i$ such that $r(\alpha_i)=i$. Consider the elements
\[
t_+=\sum_{i\in E_0}\alpha_i,\qquad t_-=t_+^*
\]
One checks that $t_-t_+=1$. Thus, since $|t_\pm|=\pm 1$, the endomorphism
\[
\phi:L\to L, \qquad \phi(x)=t_+xt_-
\]
is homogeneous of degree $0$ with respect to the $\Z$-grading. In particular it restricts to an endomorphism of $L_0$.
By \cite[Lemma 2.4]{skew}, we have
\begin{equation}\label{skewle}
L=L_0[t_+,t_-,\phi].
\end{equation}

For a unital ring $A$, we may define the Leavitt path $A$-algebra
$L_A(E)$ in the same way as before, with the proviso that elements
of $A$ commute with the generators $p_i$, $\alpha $, $\alpha ^*$.
Observe that \begin{equation} \label{tensor-A} L_A(E)=
L_{\Z}(E)\otimes A.
\end{equation}
If $A$ is a not necessarily unital ring, we take \eqref{tensor-A} as the definition of $L_A(E)$.
We may think of $L_{\Z}(E)$ as the most basic Leavitt path ring.

Let $e_0'=|\Si (E)|$. We assume that $E_0$ is ordered so that the
first $e_0'$ elements of $E_0$ correspond to its sinks. Accordingly,
the first $e_0'$ rows of the matrix $N'_E$ are $0$. Let $N_E$ be the
matrix obtained by deleting these $e_0'$ rows. The matrix that
enters the computation of the $K$-theory of the Leavitt path algebra
is
$$\begin{pmatrix}
0 \\1_{e_0-e_0'}
\end{pmatrix}-N_E^t\colon \Z^{e_0-e_0'}\longrightarrow \Z^{e_0}.$$
By a slight abuse of notation, we will write $1-N_E^t$ for this
matrix. Note that $1-N_E^t\in M_{e_0\times (e_0-e_0')}(\Z)$. Of
course $N_E=N_E'$ in case $E$ has no sinks, where $N_E'$ is
introduced in Section \ref{sect:matrix}.

\begin{thm}\label{thm:skewle}
Let $A$ be an $H'$-unital ring, $E$ a finite quiver, $M=M_E$ and $N=N_E$. Assume the quiver $E$ has no sources.
We have
\begin{align*}
K(L_{A} (E))  \cong  & NK(L_0\otimes A,\phi\otimes
1)_+\oplus NK(L_0\otimes A,\phi\otimes 1)_- \\ &\oplus
\cofi(K(A)^{e_0-e_0'}\overset{1-N^t}\longrightarrow K(A)^{e_0}).
\end{align*}
Moreover, if in addition $E$ has no sinks then
\begin{align*}
K(L_{A}(E)) \cong & NK(L_0\otimes A,\phi\otimes
1)_+\oplus NK(L_0\otimes A,\phi\otimes 1)_- \\  & \oplus
\cofi(K(A)^{e_1}\overset{1-M^t}\longrightarrow K(A)^{e_1}).
\end{align*}

\end{thm}
\begin{proof}
If $E$ has no sinks, then Proposition \ref{prop:MyN} applied to
$E^*$ gives
\[
\cofi(K(A)^{e_1}\overset{1-M^t}\longrightarrow K(A)^{e_1})\cong
\cofi(K(A)^{e_0}\overset{1-N^t}\longrightarrow K(A)^{e_0})
\]
Thus it suffices to prove the first equivalence of the theorem. By \eqref{skewle},
\[
L_{A}(E)=(L_0\otimes A)[t_+,t_-, 1\otimes\phi]
\]
Note $L_0\otimes A$ is a filtering colimit of rings of matrices with coefficients in $A$. Since $A$ is
$H'$-unital by hypothesis, each such matrix ring is $H'$-unital, whence $L_0\otimes A$ is $H'$-unital.
Hence, by Theorem \ref{thm:skewyao} \begin{align*}
K(L_{A}(E))\cong & NK(L_0\otimes A,\phi\otimes 1)_+\oplus
NK(L_0\otimes A,\phi\otimes 1)_- \\ & \oplus \cofi(K(L_0\otimes A)
\overset{1-\phi\otimes 1} \longrightarrow K(L_0\otimes A))
\end{align*}
As explained in the paragraph immediately above the theorem, we have
$L_0=\bigcup_{n=0}^\infty L_{0,n}$. Since $E$ has no sources, it
follows that $L_{0,n}$ is the product of exactly
$ne_0'+e_0=(n+1)e_0'+(e_0-e_0')$ matrix algebras; thus $K(A\otimes
L_{0,n})\cong K(A)^{(n+1)e_0'+(e_0-e_0')}$, since $A$ is $H'$-unital
and $K$-theory is matrix stable on $H'$-unital rings (by Theorem \ref{thm:excisus}).
Moreover the inclusion $L_{0,n}\subset L_{0,n+1}$ induces
$$\Delta _n:=\begin{pmatrix}  1_{(n+1)e_0'} & 0 \\ 0 & N^t
\end{pmatrix}\colon K(A)^{(n+1)e_0'+(e_0-e_0')}\longrightarrow
K(A)^{(n+1)e_0'+e_0}.$$ Now, for a path $\gamma$ on $E$, we have
$$\phi(\gamma \gamma^*)=\sum _{i,j}\alpha _i\gamma \gamma ^*\alpha
_j^*= (\alpha _{s(\gamma)}\gamma)(\alpha _{s(\gamma)}\gamma)^* ,$$
so that $\phi\otimes 1$ induces
$$\Omega _n:=\begin{pmatrix}
0 \\ 1_{ne_0'+e_0}
\end{pmatrix}\colon K(A)^{ne_0'+e_0}=K(L_{0,n}\otimes A)\longrightarrow K(A)^{(n+1)e_0'+e_0}.$$
Summing up, we get a commutative diagram
\begin{equation}
\begin{CD}
K(L_{0,n}\otimes A) @>{\Delta _n}>> K(L_{0,n+1}\otimes A) @>>>
\cdots  @>>> K(L_0\otimes A)\\
@V{\Delta _n -\Omega _n}VV @VV{\Delta _{n+1}-\Omega _{n+1}}V & &
@VV{1-\phi \otimes 1}V\\
K({L_{0,n+1}}\otimes A) @>{\Delta _{n+1}}>> K(L_{0,n+2}\otimes A)
@>>> \cdots @>>> K(L_0\otimes A)
\end{CD}
\end{equation}
Note that elementary row operations take $\Delta_n-\Omega_n$ to $1_{(n+1)e'_0}\oplus (N^t-1)$; hence there is
an elementary matrix $h$ such that $h(\Delta_n-\Omega_n)=1_{(n+1)e'_0}\oplus (N^t-1)$. Moreover one checks
that $h$ restricts to the identity on $0\oplus K(A)^{e_0}\subset K(A)^{(n+1)e'_0+e_0}$. It follows that
the inclusion $i_{n+1}:K(A)^{e_0}\to 0\oplus K(A)^{e_0}\subset K(A)^{(n+1)e'_0+e_0}$ induces an equivalence
\begin{multline*}
C:=\cofi(K(A)^{e_0-e_0'}\overset{1-N^t}\longrightarrow K(A)^{e_0})\\\cong
\cofi(K(L_{0,n}\otimes A)\overset{\Delta _n-\Omega
_n}\longrightarrow K(L_{0,n}\otimes A)),
\end{multline*}
and that furthermore, the diagram
\[
\xymatrix{K(L_{0,n}\otimes A)\ar[rr]^{\Delta _n-\Omega_n}\ar[d]_{\Omega_n}&& K(L_{0,n+1}\otimes A)\ar[d]_{\Omega_{n+1}}\ar[r] &C\ar@{=}[d]\\
K(L_{0,{n+1}}\otimes A)\ar[rr]^{\Delta _{n+1}-\Omega_{n+1}}&& K(L_{0,n+2}\otimes A)\ar[r] &C}
\]
is homotopy commutative. Hence
\[
\xymatrix{K(L_{0,n}\otimes A)\ar[rr]^{\Delta _n-\Omega_n}\ar[d]_{\Delta_n}&& K(L_{0,n+1}\otimes A)\ar[d]_{\Delta_{n+1}}\ar[r] &C\ar@{=}[d]\\
K(L_{0,{n+1}}\otimes A)\ar[rr]^{\Delta _{n+1}-\Omega_{n+1}}&& K(L_{0,n+2}\otimes A)\ar[r] &C}
\]
is homotopy commutative too. Thus $\cofi(1-1\otimes \phi:K(L_0\otimes A)\to K(L_0\otimes A))\cong C$.
\end{proof}

\section{$K$-theory of Leavitt algebras II: row-finite quivers}\label{sect:row-finite}

A quiver $E$ is said to be {\it row-finite} if for each $i\in
E_0$, the set $s^{-1}(i)=\{ \alpha \in E_1\mid s(\alpha)=i \}$ is
finite. This is equivalent to saying that the adjacency matrix
$N_E'$ of $E$ is a row-finite matrix. For a row-finite quiver $E$,
the Leavitt path algebras $L_\Z(E)$ and $L_A(E)$ are defined exactly
as in the case of a finite quiver.

Recall that a {\it complete subgraph} of a quiver $E$ is a subquiver
$F$ such that for every $v\in F_0$ either $s_F^{-1}(v)=\emptyset$ or
$s_F^{-1}(v)=s_E^{-1}(v)$. If $F$ is a complete subgraph of $E$,
then there is a natural homomorphism $L_A(F)\to L_A(E)$ (see
\cite[Lemma 3.2]{AMP}).

\begin{lem}\label{addingv}
Let $E$ be a finite quiver and let $F$ be a
subquiver of $E$ with $d=|F|$ and $d'=|\Si (F)|$. Let $A$ be a
unital ring. Suppose there is a vertex $v\in E_0\setminus F_0$ such
that $s_E^{-1}(v)\ne \emptyset$ and $r_E(s_E^{-1}(v))\subseteq F_0$.
Consider the subquiver $F'$ of $E$ with $F'_0=F_0\cup \{ v\}$,
$F'_1=F_1\cup s_E^{-1}(v)$. Then the following properties hold:
\begin{enumerate}
\item $L_A(F)$ is a full corner in $L_A(F')$. In particular $L_A(F)$ and $L_A(F')$ are Morita equivalent.
\item
$\cofi (1-N_F^t:K(A)^{d-d'}\to K(A)^{d}) \cong\\
 \cofi (1-N_{F'}^t:K(A)^{d+1-d'}\to K(A)^{d+1}).$
\end{enumerate}
\end{lem}

\begin{proof}
\noindent
\begin{enumerate}
\item Set $p=\sum _{i\in F_0}p_i\in L_A(F')$. It is easily seen that
$L_A(F)\cong pL_A(F')p$. Since $p$ is a full idempotent in $L_A(F')$, this proves (1).

\item Recall that we write $1-N_F^t$ for the $d\times (d-d')$-matrix $\begin{pmatrix} 0
\\ 1_{d-d'} \end{pmatrix}-N_F^t$. Note
that $v$ is a source in $F'$, so for every $j\in F'_0$ we have
$n^{F'}_{jv}=0$. The matrices
$$\begin{pmatrix} 0 \\ 1_{d+1-d'}   \end{pmatrix} - N_{F'}^t \quad
\text{and} \quad \begin{pmatrix} \begin{pmatrix} 0 \\ 1_{d-d'} \end{pmatrix}- N_F^t & 0 \\
0 & 1
\end{pmatrix} $$
are clearly equivalent by elementary transformations, from which the
result follows.
\end{enumerate}
\end{proof}

For a path $\gamma \in E_n$, with $n\ge 1$, we denote by $v(\gamma
)$ the set of all vertices appearing as range or source vertices of
the arrows of $\gamma $. If $i\in E_0$ is a trivial path, we set
$v(i)=\{ i \}$. Write $L_E=\{\gamma \in E_*\mid |v(\gamma )|=|\gamma
|+1 \}$, the set of paths without repetitions of vertices. Denote by
$r_{E_*}$ and $s_{E_*}$ the extensions of $r_E$ and $s_E$
respectively to the set of all paths in $E$.

Given a quiver with oriented cycles, we define a subquiver
$\tilde{E}$ of $E$ by setting $\tilde{E}_0=\{ i\in E_0\mid
r_{E_*}(i)\nsubseteq L_E \}$ and $\tilde{E}_1=\{ \alpha\in E_1\mid
s_E(\alpha )\in \tilde{E}_0 \}$. Observe that this is a well-defined
quiver because, if $s_E(\alpha )\in \tilde{E}_0$, then $r_E(\alpha
)\in \tilde{E}_0$ as well. If $E$ does not have oriented cycles,
then we define $\tilde{E}$ as the empty quiver.

\begin{lem}
\label{Ewithoutso}
 Let $E$ be a quiver. Then $\tilde{E}$ is a complete
subgraph of $E$ without sources, and if $\gamma \in E_*$ is a
non-trivial closed path then $\gamma \in \tilde{E}_*$.
\end{lem}

\begin{proof}
The result is clear in case $E$ does not have oriented cycles.
Suppose that $E$ has oriented cycles. By definition, $\tilde{E}$ is
a complete subgraph of $E$. Observe that if $i\in \tilde{E}_0$ then
$s_{E_*}^{-1}(i)\subseteq \tilde{E}_*$. Now if $\gamma \in E_*$ is a
non-trivial closed path we have $s(\gamma) = r(\gamma )\in \tilde{E}
_0$ and so $\gamma \in \tilde{E}_*$.

Pick $v\in \tilde{E}_0$. By construction there is $\gamma =\alpha _1
\cdots \alpha _m\in r_{E_*}(v)$ such that $|v(\gamma )|\le m$. Hence
there exists an index $i$ such that there is a non-trivial closed
path based on  $r_E(\alpha _i)$. Then $r_E(\alpha _i)\in
\tilde{E}_0$ and so $v\in \tilde{E}_0$. Therefore $\tilde{E}$ has no
sources.
\end{proof}

We are now ready to obtain our main general result for a row-finite quiver.

\begin{thm}\label{row-finitecase}
Let $A$ be either a ring with local units or an $H'$-unital ring
which is torsion free as a $\Z$-module, and let $E$ be a row-finite
quiver. Then there is a map
\[
\cofi (K(A)^{(E_0\setminus \Si (E))}\overset{1-N_E^t}\longrightarrow
K(A)^{(E_0)})\to K(L_A(E)),
\]
which induces a naturally split monomorphism at the level of homotopy groups
\begin{equation}\label{monohomo}
\pi_*(\cofi (K(A)^{(E_0\setminus \Si (E))}\overset{1-N_E^t}\longrightarrow
K(A)^{(E_0)})\to K_*(L_A(E)).
\end{equation}
\end{thm}
\begin{proof}
We first deal with the case of a finite quiver $E$. Set $d=|E_0|$
and $d'=| \Si (E)|$.

Consider the subquiver $F$ of $E$ given by $F_0=\tilde{E}_0\cup \Si
(E)$ and $F_1=\tilde{E}_1$. Using Lemma \ref{Ewithoutso} we see
that $F$ is a complete subgraph of $E$ such that every non-trivial
closed path on $E$ has all its arrows and vertices in $F$. Moreover
we have $\Si (F)=\Si (E)$.

Set $p=|F_0|$ and $k=d-p$. Suppose that $k>0$. In this case we will
build a chain of complete subgraphs of $E$,  $F=F^0\subset F^1\subset
\cdots \subset F^k=E$, with $| F^{i+1}_0\setminus F^i_0|=1$, and
such that the following conditions hold for every $i=0,\dots ,k-1$:

(i) $\Si (F^i)= \Si (E)$.

(ii) $L_\Z(F^i)$ is a full corner in  $L_\Z(F^{i+1})$.

(iii) \begin{align*} & \cofi (\begin{pmatrix} 0 \\
1_{p+i-d'}
\end{pmatrix} -N_{F^i}^t\colon K(A)^{p+i-d'}\longrightarrow K(A)^{p+i})\\
& \cong \cofi (\begin{pmatrix} 0 \\ 1_{p+i+1-d'}
\end{pmatrix} - N_{F^{i+1}}^t\colon K(A)^{p+i+1-d'}\longrightarrow K(A)^{p+i+1}).
\end{align*}

Suppose we have defined $F^i$ for $0\le i<k$. We are going to define
$F^{i+1}$. We first show that there is a vertex $v\in E_0\setminus
F^i_0$ such that $r_E(s_E^{-1}(v))\subseteq F^i_0$. Pick
$v_1\in E_0\setminus F^i_0$. Since $\Si (F^i)=\Si (E)$ we have that
$s_E^{-1}(v_1)\ne \emptyset $. If there exists $\alpha _1\in
s_E^{-1}(v_1)$ such that $r_E(\alpha _1 )\notin F^i_0$, set $v_2=
r_E(\alpha _1)$. Since the number of vertices in $E_0\setminus
F^i_0$ is finite, proceeding in this way we will get either a vertex
$v\in E_0\setminus F^i_0$ such that $r_E(s_E^{-1}(v))\subseteq
F^i_0$ or a path $\gamma =\alpha _1\alpha _2\cdots \alpha _m$ with
$\alpha _j\in E_1\setminus F^i_1 $ such that $r_E(\alpha _m)\in
\{r_E(\alpha _1),\dots , r_E(\alpha_{m-1}) \}$. But the latter case
cannot occur: the path $\gamma $ would not belong to $L_E$ and
consequently we would obtain  $r_E(\alpha _m)\in
\tilde{E}_0\subseteq F^i_0$, a contradiction. Therefore we put
$F^{i+1}_0=F^i_0\cup \{ v\}$ and $F^{i+1}_1= F^i_1\cup s_E^{-1}(v)$.
By construction we get (i) and that $F^{i+1}$ is a complete subgraph
of $E$, and (ii) and (iii) follow from Lemma \ref{addingv}.

Set $\ell = |\{ v\in \Si (E)\mid r_{E_*}^{-1}(v)\subseteq L_E\}|$.
Then we clearly have $K(L_A(F))\cong K(L_A(\tilde{E}))\oplus
K(A)^{\ell}$. Now by Lemma \ref{Ewithoutso} $\tilde{E}$ is a quiver
without sources. Note that $|\tilde{E}_0|-|\Si
(\tilde{E})|=(p-\ell)-(d'-\ell)=p-d'$, so from

Theorem \ref{thm:skewle} we get a decomposition
\begin{align*}
K(L_A(\tilde{E}))= & NK(L_0(\tilde{E})\otimes A,\phi\otimes
1)_+\oplus NK(L_0(\tilde{E})\otimes A,\phi\otimes 1)_-\oplus\\ &
\cofi (\begin{pmatrix} 0 \\ 1_{p-d'}
\end{pmatrix}-N_{\tilde{E}}^t \colon K(A)^{p-d'}\to K(A)^{p-\ell}).
\end{align*}
Hence
\begin{align}\label{nilterms}
K(L_A(F))\cong & K(L_A(\tilde{E}))\oplus K(A)^{\ell}\\
&\cong NK(L_0(\tilde{E})\otimes A,\phi\otimes 1)_+\oplus NK(L_0(\tilde{E})\otimes A,\phi\otimes 1)_-\nonumber\\
&\oplus \cofi (\begin{pmatrix}
0 \\ 1_{p-d'} \end{pmatrix}-N_F^t \colon K(A)^{p-d'}\to K(A)^{p}).\nonumber
\end{align}
This gives the result for $F^0=F$. Applying inductively (ii) and
(iii) to the quivers of the chain $F=F^0\subset F^1 \subset \cdots
\subset F^k=E$, and using Lemma \ref{lem:morita}, we get the
assertions of theorem for finite $E$. Let $E$ be a row-finite
quiver. By \cite[Lemma 3.2]{AMP}, $E$ is the filtered colimit of its
finite complete subgraphs. Since filtered colimits are exact,
$\cofi$ commutes with them, so we get the monomorphism in
(\ref{monohomo}). To compute the cokernel of this map, note that the
construction of the graph $\tilde{E}$ is functorial in the category
of row-finite quivers and complete graph homomorphisms. Moreover we
get $\tilde{E}=\colim \tilde{F}$, where $F$ ranges on the family of
all finite complete subquivers of $E$. For each $i\in \tilde{E}_0$
we select an arrow $\alpha _i\in \tilde{E}_1$ such that $r(\alpha
_i)=i$. This choice induces a compatible choice of arrows in the
quivers $\tilde{F}$ corresponding to finite complete subquivers $F$
of $E$. Hence, if $F^1\subseteq F^2$ are two finite complete
subquivers of $E$, then the corresponding corner-isomorphisms $\phi
^i$ on $L(\tilde{F}^i)_0$ satisfy that $\phi ^2
|_{L(\tilde{F}^1)_0}=\phi ^1$, and thus we obtain maps
$$\kappa _\pm \colon NK(L(F^1)_0\otimes A,\phi_1\otimes
1)_\pm\longrightarrow  NK(L(F^2)_0\otimes A,\phi_2\otimes 1)_\pm$$
such that the map $K(L_A(F^1))\to K(L_A(F^2))$, written in terms of
the decomposition given in Theorem \ref{thm:skewle}, is of the form
$\kappa _+\oplus \kappa _-\oplus \kappa$, where $\kappa $ is the map
between the corresponding hocofiber terms. The result follows.
\end{proof}

\begin{rem}\label{rem:coker}
The proof above shows that cokernel of the map \eqref{monohomo} can be expressed in terms of twisted nil-$K$-groups.
If $E$ is finite, the cokernel is
$NK_*(L_0(\tilde{E})\otimes A,\phi\otimes 1)_+\oplus NK_*(L_0(\tilde{E})\otimes A,\phi\otimes 1)_+$, by \eqref{nilterms}.
In the general case, it is the colimit of the cokernels corresponding to each of the finite complete subquivers.
\end{rem}

\section{Leavitt rings with regular supercoherent coefficients}\label{sec:regcoh}

In this section we will determine the $K$-theory of the Leavitt path
ring of a row-finite quiver over a regular supercoherent ring $k$.

Recall that a unital ring $R$ is said to be {\it coherent} if its
finitely presented modules form an abelian subcategory of the
category of all modules. We say that $R$ is {\it regular coherent}
if it is coherent and in addition any finitely presented module has
finite projective dimension. Equivalently $R$ is regular coherent if
any finitely presented module has a finite resolution by finitely
generated projective modules. The ring $R$ is called {\it
supercoherent} in case all polynomial rings $R[t_1,\dots ,t_p]$ are
coherent, see \cite{Gers}. Note that every Noetherian ring is
supercoherent. A more general version of regularity was introduced
by Vogel, see \cite{bihler}. We will call this concept {\it
Vogel-regularity}. For a coherent ring $R$, Vogel-regularity agrees
with regularity (\cite[Proposition 10]{bihler}). Since
Vogel-regularity is stable under the formation of polynomial rings
(\cite[Proposition 5(3)]{bihler}), it follows that $R[t_1, \dots
,t_p]$ is regular for every $p$ in case $R$ is regular
supercoherent. Observe also that any flat universal localization
$R\to R\Sigma^{-1}$ of a regular (super)coherent ring is also
regular (super)coherent. This is due to the fact that every finitely
presented $R\Sigma^{-1}$-module is induced from a finitely presented
$R$-module (\cite[Corollary 4.5]{scho}). In particular all the rings
$R[t_1,t_1^{-1},\dots ,t_p,t_p^{-1}]$ are regular supercoherent if
$R$ is regular supercoherent.

Next we will compute the $K$-theory of the Leavitt algebra of a quiver $E$ over a supercoherent coefficient
ring $k$. As a first step, we consider the case where $E$ is finite and without
sources.

\begin{prop}
\label{Dregcoh}
 Let $E$ be a finite quiver without sources and let $k$ be a
regular supercoherent ring. Let $B=\phi^{-1} L_0 $, where $L_0$ is
the homogeneous component of degree $0$ of $L_k(E)$.  Let $D=B\oplus
k$ be the $k$-unitization of $B$. Then $D$ is regular supercoherent.
\end{prop}

\begin{proof} Since the ring corresponding to $k[t_1,\dots ,t_p]$
is $D[t_1,\dots,t_p]$, it suffices to show that $D$ is regular
coherent whenever $k$ is so.

We are going to apply  \cite[Proposition 1.6]{Gers}: If $R=\colim
_{i\in I}R_i$, where $I$ is a filtering poset, the ring $R$ is a
flat left $R_i$-module for all $i\in I$, and each $R_i$ is regular
coherent, then $R$ is regular coherent.

We will show that $L_0$ is flat as a left $L_{0,n}$-module. It is
enough to show that $L_{0,n+1}$ is flat over $L_{0,n}$. Observe that
$$L_{0,n+1}=\bigoplus_{|\gamma |\le n, r(\gamma )\in \Si (E)} L_{0,n}\gamma
\gamma ^*\bigoplus \bigoplus _{|\gamma|=n+1} L_{0,n+1}\gamma \gamma
^*,$$ so that we only need to analyse the terms $L_{0,n+1}\gamma
\gamma ^*$ with $\gamma \in E_{n+1}$. Write $\gamma =\gamma_0 \alpha
$ with $\gamma _0\in E_n$ and $\alpha \in E_1$. For $v\in E_0$ set
$$Z_{v,n}=\{\beta \in E_1 \mid r(\beta )=v \text{ and there exists } \eta \in E_n
 \text{ such that } r(\eta)=s(\beta)\}.$$
For each $\beta \in Z_{v,n}$, select $\eta _{\beta}\in E_n$ such
that $r(\eta_{\beta})=s(\beta)$. Then
$$L_{0,n+1}\gamma \gamma ^*=\bigoplus _{\beta \in Z_{r(\alpha),n}}
L_{0,n}\eta_{\beta}\beta \alpha^*(\gamma_0)^* \cong \bigoplus
_{\beta \in Z_{r(\alpha),n}}L_{0,n}\eta_{\beta}(\eta_{\beta})^*.
$$
Thus $L_{0,n+1}$ is indeed projective as a $L_{0,n}$-module.

By \cite[Proposition 1.6]{Gers} we get that $L_0$ is regular
coherent. Now observe that $D=\colim  (e_iBe_i\oplus k)$, where
$e_i$ is the image of $1\in L_0$ through the canonical map $\varphi
_i\colon L_0\to B$ to the colimit. Since $e_iBe_i\cong L_0$ is
unital, we get that $e_iBe_i\oplus k\cong L_0\times k$, where
$L_0\times k$ denotes the ring direct product of $L_0$ and $k$, and
so it is regular coherent by the above. By another application of
\cite[Proposition 1.6]{Gers}, it suffices to check that
$e_{i+1}Be_{i+1}\oplus k$ is flat as a left $e_iBe_i\oplus
k$-module, which in turn is equivalent to checking that $L_0$ is flat
as a left $(1-e)k\times eL_0e$-module, where $e=\phi (1)=\sum _{i\in
E_0} \alpha _i\alpha _i ^*$. Recall that, for $i\in E_0$, $\alpha
_i\in E_1$ is such that $r(\alpha _i)=i$. We have
$L_0=(1-e)L_0\oplus eL_0$ and since $(1-e)L_0$ is flat as a left
$(1-e)k$-module, it suffices to show that $eL_0$ is flat as a left
$eL_0e$-module. Because
$$L_{0,1}\cong k^{\Si (E)}\times \prod _{i\in E_0}M_{|P(1,i)|}(\Z)$$
we see that there is a central idempotent $z$ in $L_0$ such that
$e\in zL_0$ and $e$ is a full idempotent in $zL_0$, that is
$zL_0=zL_0eL_0$. Now a standard argument shows that $eL_0$ is indeed
projective as a left $eL_0e$-module. Indeed there exists $n\ge 1$
and a finitely generated projective $L_0$-module $P$ such that
$$zL_0\oplus P\cong (L_0e)^n;$$
tensoring this with $eL_0$ we get $eL_0\oplus eP\cong
(eL_0e)^n$, as wanted. This concludes the proof.
\end{proof}

Our next lemma follows essentially from Waldhausen \cite{Wald}.

\begin{lem}
\label{lemma:Wald} Let $R$ be a regular supercoherent ring and let
$\phi$ be an automorphism of $R$. Extend $\phi$ to an automorphism of
$R[t_1,t_1^{-1},\dots ,t_p,t_p^{-1}]$ by  $\phi(t_i)=t_i$.
Then $NK_n(R[t_1,t_1^{-1},\dots,t_p,t_p^{-1}],\phi)_{\pm}=0$
for every $p\ge 0$ and every $n\in\Z$.
\end{lem}

\begin{proof}
For $n\ge 1$ this follows from \cite[Theorem 4]{Wald}, because, as
we observed before, $R[t_1,t_1^{-1},\dots ,t_p,t_p^{-1}]$ is regular
coherent. Let $n\le 1$ and assume that $NK_i(R[t_1,t_1^{-1},\dots
,t_p,t_p^{-1}],\phi)_+=0$ for every $p\ge 0$, for every $i\ge n$,
and for every automorphism $\phi$ of $R$. To show the result for
$NK_{n-1}$, it will be enough to show that $NK_{n-1}(R,\phi)_+=0$.
Since $R[t,t^{-1}]$ is regular supercoherent we have
$$K_{n}((R[t,t^{-1}])[s,\phi])=K_{n}(R[t,t^{-1}])\oplus NK_{n}(R[t,t^{-1}],\phi)=K_{n}(R[t,t^{-1}]),$$
by induction hypothesis. It follows that
\begin{equation}
\label{eq:primK-n} K_{n}(R[t,t^{-1}][s,\phi])=K_{n}(R)\oplus
K_{n-1}(R)
\end{equation}
because $NK_{n}(R)=0$ again by induction hypothesis. On the other
hand we have
\begin{align}
\label{eq:secK-n} K_{n} & (R[s,\phi]  [t,t^{-1}])
=K_{n}(R[s,\phi])\oplus K_{n-1}(R[s,\phi])\oplus
NK_{n}(R[s,\phi])^2\\
\nonumber & =K_{n}(R)\oplus K_{n-1}(R)\oplus
NK_{n-1}(R,\phi)_+\oplus NK_{n}(R[s,\phi])^2.
\end{align}
Comparison of (\ref{eq:primK-n}) and (\ref{eq:secK-n}) gives
$$NK_{n-1}(R,\phi)_+=0= NK_{n}(R[s,\phi]),$$
as desired.
\end{proof}

\begin{prop}
\label{firstregcoh} Let $k$ be a regular supercoherent ring and let
$E$ be a finite quiver without sources. Set $d=|E_0|$ and $d'=|\Si
(E)|$. Then $$K(L_k(E))\cong \cofi
(K(k)^{d-d'}\overset{1-N_E^t}\longrightarrow K(k)^{d}).$$
\end{prop}

\begin{proof}
Let $B=\phi ^{-1}L_0$, where $\phi =\phi \otimes 1\colon
L_0=L_0^{\Z}\otimes k\to L_0=L_0^{\Z}\otimes k$ is the
corner-isomorphism defined by $\phi (x)=t_+xt_-$, as in Section
\ref{sect:skew}. Note that since $k$ is regular supercoherent and
$B$ is $H'$-unital we have $NK(\tilde{B},\hat{\phi})_{\pm}=NK(B\oplus k,
\hat{\phi})_{\pm}$, where $B\oplus k$ denotes the $k$-unitization of
$B$. Now it follows from Proposition \ref{Dregcoh} that $B\oplus k$
is regular supercoherent. Therefore Lemma \ref{lemma:Wald} gives
that $NK(B\oplus k, \hat{\phi})_{\pm} =0$. It follows that
$NK(L_0,\phi)_{\pm}=NK(\tilde{B},\hat{\phi})_{\pm}=NK(B\oplus
k,\hat{\phi})_\pm =0$ and so the result follows from Theorem
\ref{thm:skewle}.
\end{proof}

\begin{thm}
\label{rf-coh} Let $k$ be a regular supercoherent ring and let $E$
be a row-finite quiver. Then
$$K(L_k(E))\cong \cofi (K(k)^{(E_0\setminus \Si (E))}\overset{1-N_E^t}\longrightarrow
K(k)^{(E_0)}).$$ It follows that there is a long exact sequence
\begin{multline*}
K_n(k)^{(E_0\setminus \Si (E))} \stackrel{1-N_E^t}\longrightarrow K_n(k)^{(E_0)}\\
\longrightarrow K_n(L_k(E))\longrightarrow K_{n-1}(k)^{(E_0\setminus \Si (E))}.
\end{multline*}
\end{thm}

\begin{proof} The case when $E$ is finite follows from Proposition \ref{firstregcoh}
and the argument of the proof of Theorem \ref{row-finitecase}. The general case follows
from the finite case, by the same argument as that given for the proof of \ref{row-finitecase}.
\end{proof}

\begin{cor}
Let $k$ be a principal ideal domain and let $E$ be a row-finite
quiver. Then
$$K_0(L_k(E))\cong \coker (1-N^t\colon \Z^{(E_0\setminus \Si
(E))}\longrightarrow \Z^{(E_0)}),$$ and
\begin{equation*}\begin{CD}K_1(L_k(E)) & \cong \coker (1-N^t\colon
K_1(k)^{(E_0\setminus \Si (E))}\longrightarrow K_1(k)^{(E_0)})\\
& \bigoplus \ker (1-N^t\colon \Z^{(E_0\setminus \Si
(E))}\longrightarrow \Z^{(E_0)}).\end{CD}\end{equation*}
\end{cor}

\begin{rem}
\label{rem:regcoh} If we only assume that $k$ is regular coherent in
Theorem \ref{rf-coh}, then the long exact sequence in the statement
terminates at $K_0(L_k(E))$, although conjecturally the long exact
sequence should still stand under this weaker hypothesis on $k$, see
\cite{bihler}.
\end{rem}

\section{Homotopy algebraic $K$-theory of the Leavitt algebra}\label{sec:kh}

Homotopy algebraic $K$-theory, introduced by C. Weibel in \cite{kh}, is a particularly well-behaved variant of algebraic $K$-theory:
it is polynomial homotopy invariant, excisive, Morita invariant, and preserves filtering colimits. There is a comparison map
\begin{equation}\label{compakh}
K_*(A)\to KH_*(A).
\end{equation}
It is proved in \cite{kh} that if $A$ is unital and $K_n(A)\to
K_n(A[t_1,\dots,t_p])$ is an isomorphism for all $p\ge 1$ (i.e. $A$
is {\it $K_n$-regular}) then \eqref{compakh} is an isomorphism for
$*\le n$. In particular if $A$ is unital and {\it $K$-regular}, that
is, if it is $K_n$-regular for all $n$,
 then \eqref{compakh} is an isomorphism for all $*\in\Z$. Further, we have:

\begin{lem}\label{lem:reghuni} Let $A$ be a $H'$-unital ring, torsion free as a $\Z$-module. If $A$ is $K_n$-regular, then $K_m(A)\to KH_m(A)$ is an isomorphism for all $m\le n$.
\end{lem}
\begin{proof}

By Remark \ref{rem:hh'}, $A[t_1,\dots,t_p]$ is $H'$-unital for all $p$. Hence
the split exact sequence of rings
\[
0 \to A[t_1,\dots,t_p]\to \tilde{A}[t_1,\dots,t_p]\to
\Z[t_1,\dots,t_p]\to 0
\]
induces a decomposition $K_*(\tilde{A}[t_1,\dots,t_p])=K_*(\Z)\oplus K_*(A[t_1,\dots,t_p])$, since $\Z$
is $K$-regular. Thus
$\tilde{A}$ is $K_n$-regular, and therefore $K_m(\tilde{A})=KH_m(\tilde{A})=KH_m(A)\oplus K_m(\Z)$ for $m\le n$. Splitting off the summand $K_m(\Z)$, we get the result.
\end{proof}

\begin{exa}\label{exa:kreg}
Examples of $K$-regular rings include regular supercoherent rings
(see \cite[Theorem 4]{Wald}), and both stable and commutative
$C^*$-algebras (see \cite[3.4, 3.5]{rosen} and \cite[5.3]{fw}). A
theorem of Vorst (see \cite{vorst}) says that if a unital ring $R$
is $K_n$-regular, then it is $K_m$-regular for all $m\le n$. If $R$
is commutative unital and of finite type over a field of
characteristic zero, then $R$ is $K_{-\dim R}$-regular
(\cite{chsw}).
\end{exa}
\begin{thm}\label{thm:khskewyao}
Let $R$ be a unital ring and let $A$ be a ring. Let
$\phi:R\to pRp$ be a corner-isomorphism. Then

\begin{align*}  KH((R\otimes A)[\tp,\tm,\phi\otimes
1]) & \cong \cofi(KH(R\otimes A)\overset{1-\phi\otimes
1}\longrightarrow KH(R\otimes A)).
\end{align*}
\end{thm}
\begin{proof}
We shall assume that $A=\Z$ and $\phi$ is an isomorphism; the general case follows from this by the same argument as in the proof of Theorem \ref{thm:skewyao}, keeping in mind that $KH$ satisfies
excision for all (not necessarily $H'$-unital) rings. By \cite[Thm. 6.6.2]{kk} there exist a triangulated category $kk$ and a functor $j:\mathrm{Rings}\to kk$
which is matrix invariant and polynomial homotopy invariant, sends short exact sequences of rings to
exact triangles, and is universal initial among all such functors. Hence the functor
$\mathrm{Rings}\to \mathrm{Ho(Spectra)}$, $A\mapsto KH(A)$, factors through an exact functor
$\overline{KH}:kk\to\mathrm{Ho(Spectra)}$. By \cite[Thm. 7.4.1]{kk}, there
is an exact triangle in $kk$
\[
\xymatrix{R\ar[r]^{1-\phi}&R\ar[r]&R[t,t^{-1},\phi]\ar[r]&\Sigma R}
\]
Applying $\overline{KH}$ we get an exact triangle
\[
\xymatrix{KH(R)\ar[r]^{1-\phi}&KH(R)\ar[r]&KH(R[t,t^{-1},\phi])\ar[r]&\Sigma KH(R)}.
\]
\end{proof}

\begin{lem}\label{lem:khmorita} Let $R$ be a unital ring, $e\in R$ an idempotent. Assume $e$ is full.
Further let $A$ be any ring. Then the inclusion map $eRe\otimes A\to R\otimes A$ induces an equivalence
$KH(eRe\otimes A)\to KH(R\otimes A)$.
\end{lem}
\begin{proof} By definition, $KH(R)=|[n]\to K(R[t_0,\dots,t_n]/<t_0+\dots+t_n-1>)|$. The case $A=\Z$ follows from \ref{lem:morita} applied to each of the polynomial rings $R[t_0,\dots,t_n]/<t_0+\dots+t_n-1>$. As in the proof of Lemma \ref{lem:morita}, the general case follows from the case $A=\Z$ by excision.
\end{proof}
\begin{thm}\label{thm:kh}
Let $A$ be a ring, and $E$ a row-finite quiver. Then
\[
KH(L_A(E))\cong \cofi (KH(A)^{(E_0\setminus \Si (E))}\overset{1-N_E^t}\longrightarrow
KH(A)^{(E_0)}).
\]
\end{thm}
\begin{proof} The case when $E$ is finite and has no sources follows from Theorem \ref{thm:khskewyao} using
the argument of the proof of Theorem \ref{thm:skewle}. The case for arbitrary finite $E$ follows as in the proof of Theorem \ref{row-finitecase}, substituting Lemma \ref{lem:khmorita} for \ref{lem:morita}. The general
case follows from the finite case by the same argument as in \ref{row-finitecase}.
\end{proof}

\begin{exa}
As an application of the theorem above, consider the case when $E$ is the quiver with one vertex and $n+1$ loops. In this case, $L_\Z(E)=L_{1,n}$ is the classical Leavitt ring \cite{lea}, and $N_E^t=[n+1]$. Hence by Theorem \ref{thm:khskewyao},
we get that $KH(A\otimes L_{1,n})$ is $KH$ with $Z/n$-coefficients:
\begin{equation}\label{khcoeff}
KH_*(A\otimes L_{1,n})=KH_*(A,\Z/n)
\end{equation}
Thus the effect on $KH$ of tensoring with $L_{1,n}$ is similar to the effect on $K^{\top}$ of tensoring a $C^*$-algebra with the Cuntz algebra $\cO_{n+1}$ (\cite{cuo}, \cite{cuo2}). If $A$ is a $\Z[1/n]$-algebra, then $KH_*(A,\Z/n)=K_*(A,\Z/n)$ \cite[1.6]{kh}, so we may substitute $K$-theory for homotopy $K$-theory in the right hand side of \eqref{khcoeff}.
\end{exa}

\section{Comparison with the $K$-theory of Cuntz-Krieger algebras}
In this section we consider the Cuntz-Krieger $C^*$-algebra $C^*(E)$
associated to a row-finite quiver $E$. If $\fA$ is any
$C^*$-algebra, we write $C_\fA^*(E)=C^*(E)\sotimes \fA$ for the
$C^*$-algebra tensor product. Since $C^*(E)$ is nuclear, there is no
ambiguity on the C$^*$-norm we are using here. Define a map
$\gamma^\fA_n=\gamma_n^\fA(E)$ so that the following diagram
commutes
\[
\xymatrix{K_n(C^*_\fA(E))\ar[r]&KH_n(C^*_\fA(E))\ar[d]\\ K_n(L_\fA(E))\ar[u]\ar[r]^(.55){\gamma^\fA_n} &K^{\top}_n(C^*_\fA(E))}
\]
The purpose of this section is to analyze when the map $\gamma^\fA_n$ is an isomorphism.
\goodbreak
The following is the spectrum-level version of a result of Cuntz and
Krieger \cite{ck}, \cite{cII}, later generalized by others; see e.g.
\cite[Theorem 3.2]{rasy}.
\begin{thm}\label{thm:ktop}
Let $\fA$ be a $C^*$-algebra and $E$ a row-finite quiver.
Then
\[
K^{\top}(C_\fA^*(E))=\cofi(\xymatrix{K^{\top}(\fA)^{(E_0\setminus\Si
E)}\ar[r]^{1-N^t_E}& K^{\top}(\fA)^{(E_0)}})
\]
\end{thm}
\begin{proof}
The proof follows the same steps as the one of Theorem \ref{thm:kh}.
In particular, the same arguments allow us to reduce to the case of
a finite quiver $E$ with no sources. In this case essentially the
same proof as in \cite[Proposition 3.1]{cII} applies. Namely, note
that $L_{\fA}(E)$ is isomorphic to a dense $*$-subalgebra of
$C_\fA^* (E)$, and let $\mathcal F$ be the norm completion of
$L_0(E)\otimes \fA$ in $C_\fA^* (E)$. Then $\mathcal K\sotimes
C_\fA^* (E)$ is a crossed product of $\mathcal K\sotimes \mathcal F$
by an automorphism $\hat{\phi}$, and Pimsner-Voiculescu  gives an
exact triangle
\begin{equation*}
\begin{CD}
\mathcal K \sotimes \mathcal F @>1-\hat{\phi}>> \mathcal K \sotimes
\mathcal F @>>> \mathcal K \sotimes C_\fA^*(E) @>>> \Sigma (\mathcal
K \sotimes \mathcal F)
\end{CD}
\end{equation*}
in KK. Now stability gives the following exact triangle in KK:
\begin{equation}
\label{PVtriangle}
\begin{CD}
\mathcal F @>1-\phi>>  \mathcal F @>>>  C_\fA^*(E) @>>> \Sigma
\mathcal F
\end{CD}
\end{equation}
where $\phi $ is just a corner-isomorphism. Since
$C^*\text{-alg}\longrightarrow \text{KK}$ is universal amongst all
stable, homotopy invariant, half-exact for cpc-split extensions
functors to a triangulated category and
$$C^*\text{-alg}\longrightarrow  \text{Ho(Spectra)},\qquad  A\mapsto K^{top}(A)$$
is one such functor which in addition maps mapping cone triangles to
exact triangles in  $\text{Ho(Spectra)}$,  the exact triangle
(\ref{PVtriangle}) is exact in Ho(Spectra); see \cite[Theorem
8.27]{cmr}. But just as in the proof of Theorem \ref{thm:skewle}, we
get
\begin{align*} &\cofi(\xymatrix{K^{\top}(\mathcal
F)\ar[r]^{1-\phi} &
K^{\top}(\mathcal F)})\\
& \cong \cofi(\xymatrix{K^{\top}(\fA)^{(E_0\setminus\Si
E)}\ar[r]^{1-N^t_E}& K^{\top}(\fA)^{(E_0)}})
\end{align*}
This concludes the proof.
\end{proof}

\begin{cor}\label{cor:kktopnn-1} Assume $K_*(\fA)\to K^{\top}_*(\fA)$ is an isomorphism for $*=n,n-1$. Then $\gamma^\fA_n$ is a split
surjection. If in addition $K_*(\fA)\to KH_*(\fA)$ and $K_*(L_\fA(E))\to KH_*(L_\fA(E))$ are isomorphisms
for $*=n,n-1$, then $\gamma_n$ is an isomorphism.
\end{cor}
\begin{proof} We have
\begin{multline*}
\pi_n \big(\cofi(\xymatrix{K(\fA)^{(E_0\setminus\Si E)}\ar[r]^(.6){1-N^t_E}& K(\fA)^{(E_0)}})\big)\\
\cong \pi_n\big( \cofi(\xymatrix{K^{\top}(\fA)^{(E_0\setminus\Si
E)}\ar[r]^(.6){1-N^t_E}& K^{\top}(\fA)^{(E_0)}})\big)
\end{multline*}
by the five lemma. Next apply Theorems \ref{row-finitecase} and \ref{thm:ktop} to obtain the first assertion.
For the second assertion, use Theorem \ref{thm:kh}.
\end{proof}

\begin{thm}\label{thm:sus}
Let $E$ be a finite quiver without sinks.
Assume that $\det(1-N_{E}^t)\ne 0$. Then $\gamma_n^\C$ is an isomorphism
for $n\ge 0$ and the zero map for $n\le -1$.
\end{thm}
\begin{proof} Because $\C$ is regular supercoherent, we have
\begin{equation}\label{Lc}
K(L_\C(E))\cong \cofi(\xymatrix{K(\C)^{(E_0)}\ar[r]^{1-N^t_E}&
K(\C)^{(E_0)}}),
\end{equation}
by Theorem \ref{rf-coh}. Thus $K_n(L_\C(E))=0$ for $n\le -1$, and
$\gamma^\C_0$ is an isomorphism by the five lemma. Moreover if
$n=|\det(1-N^t_E)|$, then $n^2K_*(L_\C(E))=0$, by \eqref{Lc}. Hence
the sequence
\begin{equation}\label{ktors}
0\to K_m(L_\C(E))\to K_m(L_\C(E),\Z/n^2)\to K_{m-1}(L_\C(E))\to 0
\end{equation}
is exact for all $m$. On the other hand, by \eqref{Lc} and Theorem \ref{thm:ktop}, we have a map of exact sequences
$(m\in\Z)$
\[
\xymatrix{K_m(\C,\Z/n^2)^{(E_0)}\ar[d]\ar[r]&K^{\top}_m(\C,\Z/n^2)^{(E_0)}\ar[d]\\
          K_m(\C,\Z/n^2)^{(E_0)} \ar[r]\ar[d]& K^{\top}_m(\C,\Z/n^2)^{(E_0)}\ar[d]\\
          K_m(L_\C(E),\Z/n^2)\ar[r]\ar[d]& K^{\top}_m(C^*_\C(E),\Z/n^2)\ar[d]\\
          K_{m-1}(\C,\Z/n^2)^{(E_0)}\ar[r]\ar[d]& K^{\top}_{m-1}(\C,\Z/n^2)^{(E_0)} \ar[d]\\
          K_{m-1}(\C,\Z/n^2)^{(E_0)} \ar[r]& K^{\top}_{m-1}(\C,\Z/n^2)^{(E_0)}}
\]
By a theorem of Suslin \cite{sus2} the comparison map $K_m(\C,\Z/q)\to K^{\top}_m(\C,\Z/q)$ is
an isomorphism for $m\ge 0$ and $q\ge 1$. Hence the map $K_*(L_\C(E),\Z/q)\to K^{\top}_*(L_\C(E),\Z/q)$
is an isomorphism, by Theorems \ref{rf-coh} and \ref{thm:ktop}. Combine this together with \eqref{ktors} and induction to finish the proof.
\end{proof}
\begin{rem} Chris Smith, a student of Gene Abrams, has given a geometric characterization of those finite quivers $E$ with no sinks which satisfy $\det(1-N_E^t)\ne 0$ \cite{smith}.
\end{rem}
\begin{exa} It follows from the theorem above that the map $\gamma_n^{\fA}$ is an isomorphism for every
finite dimensional $C^*$-algebra $\fA$. Let $\{\fA_n\to \fA_{n+1}\}_n$ be an inductive system of
finite dimensional $C^*$-algebras; write $A$ and $\fA$ for its algebraic and its $C^*$-colimit. Because $K$-theory
commutes with algebraic filtering colimits and $K^{\top}$ commutes with $C^*$-filtering colimits,
we conclude that, for $E$ as in the theorem abovem, the map $K_*(L_A(E))\to K_*(L_\fA(E))$ is an isomorphism
for $*\ge 0$.
\end{exa}

\begin{rem} Let $E$ be a finite quiver with sinks, $\tilde{E}\subset E$ as in Lemma \ref{Ewithoutso}, and
$F=\tilde{E}\cup \Si(E)$. Then, by Theorem \ref{rf-coh} and the proof of Theorem \ref{row-finitecase},
$K_n(L_\C(E)=K_n(L_\C(F))=K_n(L_\C(\tilde{E}))\oplus K_n(\C)^{\Si(E)}$. Similarly,
\[K^{\top}_n(C_\C^*(E))=K_n^{\top}(C_\C^*(\tilde{E}))\oplus K^{\top}_n(\C)^{\Si(E)}.\]
By naturality,
$\gamma_n^\C$ restricts on $K_n(\C)^{\Si(E)}$ to the direct sum of copies of the comparison map
$K_n(\C)\to K^{\top}_n(\C)$. Since the latter map is not an isomorphism for $n\ne 0$, it follows
that $\gamma_n^\C$ is not an isomorphism either.
\end{rem}

\begin{rem} It has been shown that if $\fA$ is a properly infinite $C^*$-algebra then the comparison
map $K_*(\fA)\to K_*^{\top}(\fA)$ is an isomorphism \cite{cp}. Thus $K_*(C_\C^*(E))\to K^{\top}_*(C_\C^*(E))$ is
an isomorphism whenever $C_\C^*(E)$ is properly infinite.
\end{rem}

The following proposition is a variant of a theorem of Higson (see \cite[3.4]{rosen}) that asserts that stable
$C^*$-algebras are $K$-regular.

\begin{prop}\label{prop:stablereg}
Let $A$ be an $H'$-unital ring, and $\fB$ a stable $C^*$-algebra. Then $A\otimes \fB$ is $K$-regular.
\end{prop}
\begin{proof}
By Lemma \ref{lem:Qhuni} we may assume that $A$ is a $\Q$-algebra. Since $A\to A[t]$ preserves $H$-unitality,
the proposition amounts to showing that the functor $A\mapsto K_*(A\otimes\fB)$ is invariant under polynomial
homotopy. Observe that if $\fA$ is any $C^*$-algebra, then $A\otimes(\fB\sotimes\fA)$ is $H$-unital, which
implies that the functor $A\mapsto E(A)=K_*(A\otimes(\fB\sotimes\fA))$, which is stable (because $K$-theory is
matrix stable on $H'$-unital rings), is also split exact. Hence $E$ is invariant under continuous homotopies,
by Higson's homotopy invariance theorem \cite{hig}.
Thus $E$ sends all the evaluation maps $\ev_t:\C[0,1]\to \C$ to the same map. But since the evaluation
maps $\ev_i:A[t]\to A$ factor through $\ev_i:A\otimes\C[0,1]\to A$, it
follows that $A\mapsto E(\C)=K_*(A\otimes\fB)$ is invariant under polynomial homotopies, as we had to prove.
\end{proof}
\begin{cor}\label{cor:stablereg}
If $\fB$ is a stable $C^*$-algebra and $E$ a row-finite quiver, then both $\fB$ and $L_\fB(E)$ are $K$-regular,
and the map of Theorem \ref{row-finitecase}
\[
\cofi (K(\fB)^{(E_0\setminus \Si (E))}\overset{1-N_E^t}\longrightarrow
K(\fB)^{(E_0)})\to K(L_\fB(E))
\]
is an equivalence.
\end{cor}
\begin{proof}
That $\fB$ and $L_\fB(E)$ are $K$-regular is immediate from the
proposition; by Corollary \ref{cor:otimes}, they are also
$H$-unital. It follows from this and from Lemma \ref{lem:reghuni}
that the comparison maps $K(\fB)\to KH(\fB)$ and $K(L_\fB(E))\to
KH(L_\fB(E))$ are equivalences. Now apply Theorem \ref{thm:kh}.
\end{proof}
\begin{thm}\label{thm:stable}
If $\fB$ is a stable $C^*$-algebra then the map $\gamma_n^\fB$ is an isomorphism for every $n$ and
every row-finite quiver $E$.
\end{thm}
\begin{proof}
The theorem is immediate from Corollary \ref{cor:stablereg}, Theorem \ref{thm:ktop}, and the fact
(proved in \cite{kar} for $n\le 0$ and in \cite{suswod} for $n\ge 1$) that the map
$K_n(\fB)\to K^{\top}_n(\fB)$ is an isomorphism for all $n$.
\end{proof}

\begin{rem} If $\fB$ is stable, then $C_\fB^*(E)$ is stable, and thus the comparison map
$K_*(C_\fB^*(E))\to K^{\top}_*(C_\fB^*(E))$ is an isomorphism. Moreover we also have
$KH(C_\fB^*(E))\cong K^{\top}_*(C_\fB^*(E))$, by \ref{prop:stablereg}.
\end{rem}

\begin{ack} Part of the research for this article was carried out during visits of the third named author
to the Centre de Recerca Matem\`atica and the Departament de Matem\`atiques of the Universitat Aut\`onoma
de Barcelona. He is indebted to these institutions for their hospitality.
\end{ack}


\begin{thebibliography}{20}

\bibitem{AA}
G. Abrams, G. Aranda Pino. {\it The Leavitt path algebra of a graph.} J.
Algebra {\bf 293} (2005), 319--334.

\bibitem{fpres}
P. Ara, M. Brustenga. {\it Module theory over Leavitt path
algebras and $K$-theory}.  Preprint ~2009.

\bibitem{skew}
P. Ara, M.A. Gonz\'alez-Barroso, K.R. Goodearl, E. Pardo.
{\it  Fractional
skew monoid rings} J.  Algebra {\bf 278} (2004) 104--126.

\bibitem{AMP}
P. Ara, M.A. Moreno, E. Pardo.{\it Nonstable K-theory for graph
algebras}. Algebras Represent. Theory {\bf 10} (2007) 157--178.

\bibitem{bihler}
F. Bihler. {\it Vogel's notion of regularity for non-coherent
rings}. arXiv:math/0612569v1 [math.KT].


\bibitem{kabi}
G. Corti\~nas. {\it The obstruction to excision in K-theory and in cyclic
homology}. Invent. Math. {\bf 454} (2006) 143--173.

\bibitem{friendly}
G. Corti\~nas.{\it Algebraic vs. topological $K$-theory: a friendly
match.} Preprint. Available at
http://mate.dm.uba.ar/\~{}gcorti/friendly.pdf.

\bibitem{cp}
G. Corti\~nas, N.C. Phillips.
{\it Algebraic $K$-theory and properly infinite $C^*$-algebras.} Preprint.

\bibitem{chsw}
G. Corti\~nas, C. Haesemeyer, M. Schlichting, C. Weibel.
{\it  Cyclic homology, cdh-cohomology and negative K-theory}. Ann. of Math. {\bf 167}, (2008) 549--573.

\bibitem{kk}
G. Corti\~nas, A. Thom. {\it Bivariant algebraic K-theory}. J. reine angew. Math. 510, 71--124.

\bibitem{cuo}
J. Cuntz. {\it $K$-theory for certain $C\sp{\ast} $-algebras.} Ann. Math. (2) {\bf 113}
(1981), no. 1, 181--197.

\bibitem{cuo2}
J. Cuntz. {\it $K$-theory for certain $C\sp{\ast} $-algebras II.} J. Operator
Theory  {\bf 5 }  (1981), no. 1, 101--108.

\bibitem{cII}
J. Cuntz. {\it A class of $C\sp{\ast} $-algebras and topological
Markov chains II: reducible chains and the Ext-functor for
$C^*$-algebras}. Invent. Math. 63 (1981), no. 1, 25--40.


\bibitem{ck}
J. Cuntz, W. Krieger. {\it A class of $C\sp{\ast} $-algebras and
topological Markov chains}. Invent. Math. 56 (1980), no. 3,
251--268.


\bibitem{cmr} J. Cuntz, R. Meyer, J. Rosenberg. {\it Topological and bivariant $K$-theory}.
Oberwolfach Seminars, 36. Birkh\"{a}user Verlag, Basel, 2007.


\bibitem{fw}
E. Friedlander, M.E. Walker.
{\it Comparing $K$-theories for complex varieties.}
Amer. J. of Math. {\bf 128} (2001) 779--810.

\bibitem{Gers}
S. M. Gersten.{\it $K$-theory of free rings}. Comm. in Algebra {\bf 1}
(1974) 39--64.


\bibitem{gray}
D. Grayson. {\it The $K$-theory of semilinear endomorphisms}. J. Algebra
{\bf 113} (1988), 358--372.

\bibitem{hig}
 N. Higson. {\it Algebraic $K$-theory of $C^*$-algebras.} Adv. in Math. {\bf 67}, (1988) 1--40.

\bibitem{kar}
M. Karoubi.
{\it $K$-th\'eorie alg\'ebrique de certaines alg\`ebres d'op\'erateurs}.
Alg\`ebres d'op\'erateurs (S\'em., Les Plans-sur-Bex, 1978), pp. 254--290,
Lecture Notes in Math., 725, Springer, Berlin, 1979.

\bibitem{lea}
W. G. Leavitt. {\it The module type of a ring.}
 Trans. Amer. Math. Soc. {\bf 103} (1962) 113--130.


\bibitem{rasy}  I. Raeburn, W. Szyma\'nski.
{\it Cuntz-Krieger algebras of infinite graphs and matrices} Trans.
Amer. Math. Soc. {\bf 356} (2004) 39--59.


\bibitem{rosen} J. Rosenberg.
 {\it Comparison between algebraic and topological $K$-theory for Banach algebras and
$C^*$-algebras.} In Handbook of K-Theory, Friedlander, Eric M.; Grayson, Daniel R. (Eds.). Springer-Verlag, New York, 2005.

\bibitem{scho} A. H. Schofield, {\it Representations of Rings
over Skew Fields}, LMS Lecture Notes Series 92, Cambridge Univ.
Press, Cambridge, UK, 1985.

\bibitem{smith}
C. Smith.{\it Unpublished notes, 2008}.

\bibitem{sus}
A. Suslin. {\it Excision in the integral algebraic $K$-theory}.
Proceedings of the Steklov Institute of Mathematics {\bf 208} (1995)
255--279.

\bibitem{sus2}
A. Suslin.{\it Algebraic $K$-theory of fields}.
Proceedings of the International Congress of Mathematicians, Vol. 1, 2 (Berkeley, Calif., 1986),
222--244, Amer. Math. Soc., Providence, RI, 1987.

\bibitem{suswod}
A. Suslin, M. Wodzicki.{\it Excision in Algebraic K-Theory}. Ann. of
Math. (2) {\bf 136} (1992), 51--122.

\bibitem{vorst} T. Vorst. {\it Localization of the $K$-theory of polynomial extensions}.
Math. Annalen {\bf 244} (1979) 33--54.

\bibitem{Wald}
F. Waldhausen.{\it Algebraic $K$-theory of generalized free products I}.
Ann. of Math. (2) {\bf 108} (1978), 135--204.

\bibitem{Wei}
C. Weibel. The K-book: An introduction to algebraic $K$-theory.
Available at  http://www.math.rutgers.edu/~weibel/Kbook.html.

\bibitem{kh}
C. Weibel. {\it Homotopy Algebraic $K$-theory}. Contemporary Math. {\bf 83} (1989)
461--488.

\bibitem{wod}
M. Wodzicki.{\it Excision in cyclic homology and in rational algebraic
$K$-Theory}. Ann. of Math. (2) {\bf 129} (1989), 591--639.

\bibitem{yao}
D. Yao. {\it A note on the $K$-theory of twisted projective lines and
twisted Laurent polynomial rings}. J. Algebra {\bf 123} (1995)
424--435.
\end{thebibliography}
\end{document}